\documentclass[12pt]{article}
\usepackage{amsmath,mathrsfs}
\usepackage{amsfonts}
\usepackage{amssymb}
\usepackage{amscd}
\usepackage{amsthm}
\usepackage{latexsym}
\usepackage{amsbsy}
\usepackage{amsfonts, amsmath, amssymb, amsgen, amsthm, amscd,latexsym}
\usepackage[all]{xy}
\usepackage{amsxtra}

\def\End{\mathop{\rm End}\nolimits}
\def\Aut{\mathop{\rm Aut}\nolimits}

\def\Id{\mathop{\rm Id}\nolimits}

\def\ad{\mathop{\rm ad}\nolimits}

\def\Hom{\mathop{\rm Hom}\nolimits}

\def\Cb{{\mathbb C}}

\def\Zb{{\mathbb Z}}

\def\Hc{{\cal H}}

\def\Lc{{\cal L}}
\def\Mc{{\cal M}}

\def\Lc{{\cal L}}

\def\a{\alpha}
\def\b{\beta}
\def\d{\delta}
\def\D{\Delta}

\def\Om{\Omega}
\def\s{\sigma}

\def\t{\theta}

\def\ve{\varepsilon}
\def\vp{\varphi}

\def\0b{\bf 0}

 \def\Zb{\mathbb{Z}}

\def\ot{\otimes}

\def\ra{\rightarrow}

\def\al{>\hspace{-4pt}\vartriangleleft}

\def\p{\partial}

\def\0D{\Delta^{(0)}}
\def\1D{\Delta^{(1)}}
\def\Db{\blacktriangledown}

\newcommand{\FD}{\mathfrak{D}}

\newcommand{\Fg}{\mathfrak{g}}

\newcommand{\Fl}{\mathfrak{l}}

\newcommand{\Fs}{\mathfrak{s}}

 \newcommand{\ie}{{\it i.e., }\ }

\newtheorem{theorem}{Theorem}[section]
\newtheorem{remark}[theorem]{Remark}
\newtheorem{proposition}[theorem]{Proposition}
\newtheorem{lemma}[theorem]{Lemma}
\newtheorem{corollary}[theorem]{Corollary}

\newtheorem{example}[theorem]{Example}
\newtheorem{definition}[theorem]{Definition}

\def\build#1_#2^#3{\mathrel{
\mathop{\kern 0pt#1}\limits_{#2}^{#3}}}
\newcommand{\ps}[1]{~\hspace{-4pt}_{^{(#1)}}}

\newcommand{\ns}[1]{~\hspace{-4pt}_{_{{\langle#1\rangle}}}}

\newcommand{\snsb}[1]{~\hspace{-4pt}_{_{{[\overline{#1}]}}}}

\newcommand{\nsb}[1]{~\hspace{-4pt}_{^{[#1]}}}

\def\odots{\ot\dots\ot}
\def\wdots{\wedge\dots\wedge}

\def\wg{\wedge}

\def\vp{\varphi}

\def\b{\beta}
\def\D{\Delta}

\def\d{\delta}

\def\d{\delta}

\def\dt{\left.\frac{d}{dt}\right|_{_{t=0}}}

\numberwithin{equation}{section}
\begin{document}

\title{\bf \sc Cyclic cohomology of Lie algebras}

\author{
\begin{tabular}{cc}
Bahram Rangipour \thanks{Department of Mathematics  and   Statistics,
     University of New Brunswick, Fredericton, NB, Canada}\quad and \quad  Serkan S\"utl\"u $~^\ast$
      \end{tabular}}

\date{ \ }

\maketitle
\abstract{ In this paper we aim to understand the category of stable-Yetter-Drinfeld modules over enveloping algebra of  Lie algebras. To do so, we need to define such modules over Lie algebras.  These two categories  are shown to be isomorphic.  A mixed complex is defined for   a given   Lie algebra and a  stable-Yetter-Drinfeld module over it. This complex is quasi-isomorphic to the Hopf cyclic complex of  the enveloping algebra of the Lie algebra with coefficients in  the corresponding module.    It is shown that  the (truncated) Weil algebra, the Weil algebra with generalized coefficients defined by   Alekseev-Meinrenken, and the perturbed  Koszul complex introduced  by  Kumar-Vergne are examples of  such a mixed complex.}

\section{Introduction}
One of the  well-known complexes in mathematics is the Chevalley-Eilenberg complex of a Lie algebra $\Fg$ with coefficients in a $\Fg$-module $V$ \cite{ChevEile}.
\begin{equation}
\xymatrix{
C^\bullet(\Fg,V):&&  V\ar[rr]^{d_{\rm CE}}&& V\ot\Fg^\ast\ar[rr]^{d_{\rm CE}}&& V\ot\wedge^2\Fg^\ast\ar[rr]^{d_{\rm CE}}&& \cdots
  }
\end{equation}

Through  examples,  we can see  that when the coefficients space  $V$ is equipped with more  structures,   then  the complex  $(C^\bullet(\Fg,V),d_{\rm CE})$, together with another operator  $d_{\rm K}: C^\bullet(\Fg,V)\ra C^{\bullet-1}(\Fg,V)$,  called Koszul boundary, turns into a mixed complex. That is  $d_{\rm CE}+d_{\rm K}$ defines  a coboundary  on the  total complex $$W^\bullet=\bigoplus_{\bullet\ge p\ge 0} C^{2p-\bullet}(\Fg,V).$$  Among examples,  one observes that,
\begin{itemize}
\item the well-known (truncated) Weil complex is achieved  by $V:=S(\Fg^\ast)_{[2q]}$ the (truncated) polynomial algebra of $\Fg$,
\item the Weil algebra with generalized coefficients  defined by Alekseev-Meinrenken in \cite{AlekMein} is obtained by  $V:=\mathcal{E}'(\Fg^*)$,  the convolution algebra of compactly supported distributions on $\Fg^*$,
\item finally it was shown by  Kumar-Vergne that  if $V$ is a module over the Weyl algebra $D(\Fg)$ then $ (W^\bullet, d_{\rm CE}+d_{\rm K})$ is a complex which is called perturbed Koszul complex \cite{KumaVerg}.

\end{itemize}

In this paper we  prove that  $(W^\bullet, d_{\rm CE}+d_{\rm K})$ is a complex if and only if $V$ is a unimodular stable module over the Lie algebra $\widetilde \Fg$, where  $\widetilde\Fg:=\Fg^\ast\al\Fg$ is  the semidirect product  Lie algebra  $\Fg^\ast$ and $\Fg$. Here   $\Fg^\ast:=\Hom(\Fg, \Cb)$  is thought of as an abelian Lie algebra acted upon by the Lie algebra $\Fg$ via the coadjoint representation.

Next, we show that any  Yetter-Drinfeld module over the enveloping Hopf algebra $U(\Fg)$ yields a module over $\widetilde\Fg$ and conversely any locally conilpotent module over $\widetilde\Fg$  amounts to a Yetter-Drinfeld module over the Hopf algebra $U(\Fg)$. This correspondence is accompanied with  a quasi-isomorphism  which reduces to  the antisymmetrization map if the module $V$ is merely a $\Fg$-module. The isomorphism generalizes the computation  of the Hopf cyclic cohomology of $U(\Fg)$ in terms of  the Lie algebra homology of $\Fg$ carried out by  Connes-Moscovici in \cite{ConnMosc98}.

Throughout  the paper,  $\Fg$ denotes a finite dimensional Lie algebra over $\Cb$, the field of complex numbers. 
We denote by $X_1, \ldots,  X_N$ and $\t^1, \ldots, \t^N$ a dual basis for $\Fg$ and $\Fg^\ast$ respectively. 
All tensor products are over $\Cb$. 
\bigskip

B. R.  would like to thank Alexander Gorokhovsky for the useful discussions  on the $G$-differential algebras,  and  is also grateful to the organizers of NCGOA 2011 at Vanderbilt University, where these discussions took place. 

\tableofcontents

\section{The model complex for $G$-differential algebras}
In  this section we first recall $G$-differential algebras and their basic properties. Then we introduce our model complex which is the main motivation of this paper. The model complex includes as examples Weil algebra and their truncations, perturbed Koszul complex introduced by Kumar- Vergne in \cite{KumaVerg},
and Weil algebra with generalized coefficients introduced by Alekseev-Meinrenken \cite{AlekMein}.

\subsection{$G$-differential algebras}

  Let $\widehat{g} = \Fg_{-1}\oplus \Fg_0\oplus \Fg_1$ be a   graded Lie algebra, where $\Fg_{-1}$ and $\Fg_0$ are  $N$-dimensional vector spaces with bases  $\iota_1, \cdots , \iota_N$, and $\Lc_1, \cdots , \Lc_N$ respectively, and  $\Fg_1$ is generated by  $d$.

    We let $C^i_{jk}$ denote the  structure constants of the Lie algebra $\Fg_0$ and assume that  the  graded-bracket on $\widehat{g}$ is defined as follows.
\begin{align}
& [\iota_p,\iota_q] = 0, \\
& [\Lc_p,\iota_q] = C^r_{pq}\iota_r, \\
& [\Lc_p,\Lc_q] = C^r_{pq}\Lc_r ,\\
& [d,\iota_k]= \Lc_k ,\\
& [d,\Lc_k] = 0, \\
& [d,d] = 0.
\end{align}

Now let $G$ be a (connected) Lie group with Lie algebra $\Fg$. We assume $\widehat \Fg$ be as above with $ \Fg_0\cong\Fg$ as Lie algebras.

A graded algebra $A$ is called a $G$-differential algebra if there exists a representation $\rho:G \to \Aut(A)$ of the group $G$ and a graded  Lie algebra homomorphism $\hat{\rho}:\widehat{g} \to \End(A)$ compatible in the following way:
\begin{align}
& \dt\rho(exp(tX)) = \hat{\rho}(X) \\[.2cm]
& \rho(a)\hat{\rho}(X)\rho(a^{-1}) = \hat{\rho}(Ad_aX) \\[.2cm]
& \rho(a)\iota_X\rho(a^{-1}) = \iota_{Ad_aX} \\[.2cm]
& \rho(a)d\rho(a^{-1}) = d
\end{align}
for any $a \in G$ and any $X \in \Fg$. For further discussion on $G$-differential algebras we refer the reader to \cite[chpter 2]{GuilSter} and \cite{AlekMein}.

 The exterior algebra $\bigwedge\Fg^\ast$ and the Weil algebra  are examples of $G$-differential algebras.

 Here we recall $W(\Fg)$,  the Weil algebra  of a finite dimensional  Lie algebra $\Fg$,  by
$$W(\Fg) = \bigwedge\Fg^* \ot S(\Fg^*), $$ with the grading
\begin{align}
W(\Fg) = \bigoplus_{l \geq 0} W^l(\Fg),
\end{align}
where
\begin{align}
W^l(\Fg) = \bigoplus_{p + 2q = l} W^{p,q}, \qquad W^{p,q} := \wedge^p\Fg^* \ot S^q(\Fg^*).
\end{align}
It is equipped with two degree $+1$ differentials as follows. The first one is
\begin{align}
\begin{split}
& d_{\rm K}:\wedge^p\Fg^* \ot S^q(\Fg^*) \to \wedge^{p-1}\Fg^* \ot S^{q+1}(\Fg^*) \\
& \vp \ot R \mapsto \sum_j \iota_{X_j}(\vp) \ot R\theta^j
\end{split}
\end{align}
and it is called the Koszul coboundary. The second one is the Chevalley-Eilenberg coboundary (Lie algebra cohomology coboundary)
\begin{equation}
d_{\rm CE}:\wedge^p\Fg^* \ot S^q(\Fg^*) \to \wedge^{p+1}\Fg^* \ot S^q(\Fg^*) \\
\end{equation}

Then $d_{\rm CE} + d_{\rm K}: W^l(\Fg) \to W^{l+1}(\Fg)$ equips $W(\Fg)$ with a differential graded algebra structure.  It is known that via coadjoint representation $W(\Fg)$ is a $G$-differential algebra.

\medskip

A $G$-differential algebra is called locally free if there exists an element $$\Theta = \sum _iX_i \ot \theta^i \in (\Fg \ot A^{odd})^G$$ called the algebraic connection form.

We  assume that $\Theta \in (\Fg \ot A^1)^G$, and we have  $$\iota_k(\Theta) = X_k, \quad \text{and }\quad \Lc_k(\theta^i) = -C^i_{kl}\theta^l.$$

\subsection{The model complex}
Let $(A,\Theta)$ be a locally free $G$-differential algebra with $\dim(G)=N$. We assume that  $V$ is a vector space with elements $L_k$ and $L^k$ in  $\End(V)$, $1 \leq k \leq N$.

We consider the graded space $A \ot V$ with the grading induced from that of $A$. Using all information of  the $G$-differential algebra structure of $A$ and the connection form $\Theta \in (\Fg \ot A^1)^G$, we introduce the following map as a sum of a degree $+1$ map and a degree $-1$ map.

\begin{equation}\label{model-complex}
D(x \ot v)  := d(x) \ot v + \theta^kx \ot L_k(v) + \iota_k(x) \ot L^k(v)
\end{equation}

\begin{proposition}\label{Koszul-diff}
Let $(A, \Theta)$ be a locally free $G$-differential algebra. Then the map
\begin{equation}
d_{\rm K}(x \ot v) = \iota_k(x) \ot L^k(v)
\end{equation}
is  a  differential, that is $d_{\rm K}^2=0$,  if and only if $V$ is  a $\Cb^N$-module via $L^k$s, \ie
$$[L^p,L^q] = 0, \qquad 1 \leq p,q, \leq N.$$
\end{proposition}

\begin{proof}
Assume that $[L^j,L^i] = 0$. Then
\begin{equation}
d_{\rm K} \circ d_{\rm K}(x \ot v) = \iota_l\iota_k(x) \ot L^lL^k(v) = 0
\end{equation}

by the commutativity of $L_k$s and the anti-commutativity of $\iota_k$s.

Conversely, if $d_{\rm K}$ has the property $d_{\rm K} \circ d_{\rm K} = 0$, then by using $\iota_k(\t^j)=\d^j_k$ we have

\begin{equation}
d_{\rm K} \circ d_{\rm K}(\theta^i\theta^j \ot v) = d_{\rm K}(\theta^j \ot L^i(v) - \theta^i \ot L^j(v)) = 1 \ot [L^j,L^i](v) = 0
\end{equation}

which implies $[L^j,L^i] = 0$.
\end{proof}

\begin{definition}
\cite{GuilSter}. For a commutative locally free $G$-differential algebra $A$, the element $\Om=\sum_i \Om^i\ot X_i \in (A^2\ot \Fg)^G$,  satisfying

\begin{equation}
d(\theta^i) = -\frac{1}{2}C^i_{pq}\theta^p\theta^q + \Om^i,
\end{equation}

is  called the curvature of the connection $\Theta=\sum_i\t^i\ot X_i$.
\end{definition}

We call a commutative locally free $G$-differential algebra $(A,\Theta)$  flat if $\Om=0$, or  equivalently
\begin{equation}
d(\theta^k) = -\frac{1}{2}C^k_{pq}\theta^p\theta^q.
\end{equation}

\begin{proposition}\label{CE-diff}
Let $(A, \Theta)$ be a commutative locally free flat $G$-differential algebra. Then the map
\begin{equation}
d_{\rm CE}(x \ot v) = d(x) \ot v + \theta^kx \ot L_k(v)
\end{equation}
is a differential, that is $d_{\rm CE}^2=0$, if and only if $V$ is a $\Fg$-module via $L_k$, that is $[L_t,L_l] = C^k_{tl}L_k$.
\end{proposition}

\begin{proof}
 Using the commutativity of $A$ we see that
\begin{align}
\begin{split}
& d_{\rm CE} \circ d_{\rm CE}(1 \ot v) = \sum _k d_{\rm CE}(\theta^k \ot L_k(v)) = \sum_k d(\theta^k) \ot L_k(v) +
\sum_{k,t}\theta^t\theta^k \ot L_tL_k(v) \\
& = -\sum_{l,t}\frac{1}{2}C^k_{tl}\theta^t\theta^l \ot L_k(v) + \sum_{t,l} \frac{1}{2}\theta^t\theta^l \ot [L_t,L_l](v),
\end{split}
\end{align}
which proves the claim.
\end{proof}

\begin{proposition}\label{14}
Let $A$ be a commutative locally free flat $G$-differential algebra and $V$ be a $\Fg$-module via $L_k$s and a $\Cb^N$-module via $L^k$s. Then, $(A \ot V, D)$ is a complex with differential

\begin{equation}
D(x \ot v)  := d(x) \ot v + \theta^kx \ot L_k(v) + \iota_k(x) \ot L^k(v),
\end{equation}
if and only if
\begin{equation}
\mathcal{L}_k(x) \ot L^k(v) + \theta^k\iota_t(x) \ot [L_k,L^t](v) + x \ot L^kL_k(v) = 0.
\end{equation}
\end{proposition}

\begin{proof}
By Proposition \ref{Koszul-diff} and Proposition \ref{CE-diff},  $d_{\rm K}$ and $d_{\rm CE}$ are differentials respectively. Then  $A \ot V$ is a complex with differential $D = d_{\rm CE} + d_{\rm K}$ if and only if $d_{\rm CE} \circ d_{\rm K} + d_{\rm K} \circ d_{\rm CE} = 0$.

We observe

\begin{align*}
& d_{\rm K}(d_{\rm CE}(x \ot v)) = d_{\rm K}(d(x) \ot v + \theta^kx \ot L_k(v)) \\
& = \iota_kd(x) \ot L^k(v) + \iota_t(\theta^kx) \ot L^tL_k(v) \\
& = \iota_kd(x) \ot L^k(v) + x \ot L^kL_k(v) - \theta^k\iota_t(x) \ot L^tL_k(v),
\end{align*}
and
\begin{align}
\begin{split}
& d_{\rm CE}(d_{\rm K}(x \ot v)) = d_{\rm CE}(\iota_t(x) \ot L^t(v)) \\
& = d\iota_t(x) \ot L^t(v) + \theta^k\iota_t(x) \ot L_kL^t(v).
\end{split}
\end{align}
Therefore,
\begin{equation}
d_{\rm CE} \circ d_{\rm K} + d_{\rm K} + d_{\rm CE}(x \ot v) = \mathcal{L}_k(x) \ot L^k(v) + \theta^k\iota_t(x) \ot [L_k,L^t](v) + x \ot L^kL_k(v)
\end{equation}
\end{proof}

The next proposition determines  the conditions on $V$ that is necessary and sufficient for  $(A \ot V,d_{\rm CE}+d_{\rm K})$  to be  a complex.

Considering the dual $\Fg^*$ of the Lie algebra $\Fg$ as a commutative Lie algebra, we can define the Lie  bracket  on $\widetilde\Fg\; :=\; \Fg^\ast \al \Fg$ by
\begin{equation}\label{tdg}
\big[\a \oplus X\;,\; \b \oplus Y\big] \;:= \;\big(\Lc_X(\b) - \Lc_Y(\a)\big) \oplus \big[X\;,\;Y\big].
\end{equation}

\begin{proposition}\label{22}
Let $A$ be a commutative locally free flat $G$-differential algebra and $V$ a $\Fg$-module via $L_k$s and a $\Cb^N$-module via $L^k$s.  Then, $(A \ot V, D)$ is a  complex if and only if
\begin{equation}\label{15}
\text{unimodular stability}\qquad\qquad \sum_k L^k L_k = 0,
\end{equation}
and
\begin{equation}\label{16}
 \text{$\widetilde\Fg$-module property}\qquad\qquad  [L^i,L_j] = \sum_k C^i_{jk}L^k.
\end{equation}
\end{proposition}

\begin{proof}

Assume first that $(A \ot V, d_{\rm CE} + d_{\rm K})$ is a differential complex. Then by Proposition \ref{14}, taking $x = 1$ we get

\begin{equation*}
\sum_k L^k L_k = 0.
\end{equation*}

On the other hand,  by taking $x = \theta^i$ and using the fact that  $A$ is locally free, we get
\begin{equation*}
[L^i,L_j] = C^i_{jk}L^k.
\end{equation*}
Conversely, if \eqref{15} and \eqref{16} hold, then
\begin{equation*}
\big(d_{\rm CE} \circ d_{\rm K} + d_{\rm K} \circ d_{\rm CE}\big)(x \ot v) = \mathcal{L}_k(x) \ot L^k(v) + C^s_{kl}\theta^l\iota_s(x) \ot L^k(v) = 0.
\end{equation*}
\end{proof}

\section{Lie algebra cohomology and Perturbed Koszul complex}
In this section we specialize the model complex  $(A\ot V, D)$ defined in \eqref{model-complex} for $A=\bigwedge \Fg^\ast$. We show that the perturbed Koszul complex defined in \cite{KumaVerg} is an example of the model complex. As another example of the model complex,  we cover the Weil algebra with generalized coefficients introduced in \cite{AlekMein}.

\subsection{Lie algebra  cohomology}
Let $\Fg $ be a finite dimensional Lie algebra and $V$  be a right $\Fg$-module. Let also  $\{\t^i\}$ and $\{X_i\}$ be dual bases for $\Fg^\ast$ and $\Fg$. The Chevalley-Eilenberg complex $C(\Fg,M)$ is defined by
\begin{equation}
\xymatrix{V\ar[r]^{d_{\rm CE}\;\;\;\;\;\;\;\;}&C^1(\Fg,V)\ar[r]^{d_{\rm CE}}& C^2(\Fg,V)\ar[r]^{\;\;\;\;\;\;d_{\rm CE}}&\cdots\;,}
\end{equation}
where $C^q(\Fg,V)=\Hom(\wedge^q \Fg ,V)$ is the vector space of all alternating linear maps on $\Fg^{\ot q}$ with values in $V$. If $\a \in C^q(\Fg,V)$, then
 \begin{align}
 \begin{split}
& d_{\rm CE}(\a)(X_0, \ldots,X_q)=\sum_{i<j} (-1)^{i+j}\a([X_i,X_j], X_0\ldots \widehat{X}_i, \ldots, \widehat{X}_j, \ldots, X_q)+\\
& ~~~~~~~~~~~~~~~~~~~~~~~\sum_{i}(-1)^{i+1}\a(X_0,\ldots,\widehat{X}_i,\ldots X_q)X_i.
\end{split}
 \end{align}
Alternatively, we may identify  $C^q(\Fg,V)$ with $\wg^q\Fg^\ast \ot V$ and the coboundary $d_{\rm CE}$ with the following one
\begin{align}\label{d-CE}
\begin{split}
& d_{\rm CE}(v)= - \t^i \ot v \cdot X_i,\\
& d_{\rm CE}(\b \ot v)= d_{\rm dR}(\b)\ot v  - \t^i \wg \b \ot v \cdot X_i.
\end{split}
\end{align}
where $d_{\rm dR}:\wedge^p\Fg^\ast\ra \wedge^{p+1}\Fg^\ast$ is the de Rham  derivation defined by $d_{\rm dR}(\t^i)=\frac{-1}{2}C^i_{jk}\t^j\t^k$. We denote the cohomology of $(C^\bullet(\Fg,V),d_{\rm CE})$ by $H^\bullet(\Fg,V)$ and refer to it as the Lie algebra cohomology of $\Fg$ with coefficients in $V$.

\subsection{Perturbed Koszul complex}
With the same assumptions for $\Fg$ and $V$ as in the previous subsection, we specialize the model complex  $A\ot V$ defined in \eqref{model-complex} for $A=\bigwedge \Fg^\ast$. Indeed we have  $W^n(\Fg,V) := \wedge^n \Fg^* \ot V$, for $n \geq 0$, with differentials $d_{\rm CE}: W^n(\Fg,V) \to W^{n+1}(\Fg,V)$ defined in \eqref{d-CE}  and

\begin{align}
\begin{split}
& d_{\rm K}:W^n(\Fg,V) \to W^{n-1}(\Fg,V) \\
& \alpha \ot v \mapsto \sum_i \iota_{X_i}(\alpha) \ot v \lhd \theta^i.
\end{split}
\end{align}

Considering
\begin{equation}
L_kv = X_k \cdot v, \quad L^kv = v \lhd \theta^k,
\end{equation}

the condition $\sum_k L^kL_k = 0$ transfers directly into

\begin{equation}\label{21}
\sum_k (v \cdot X_k) \lhd \theta^k = 0.
\end{equation}

Similarly, the condition $[L^i,L_j] = \sum_k C^i_{jk}L^k$ becomes
\begin{equation}\label{3}
(v \cdot X_j) \lhd \theta^t = v \lhd (X_j \rhd \theta^t) + (v \lhd \theta^t) \cdot X_j.
\end{equation}

\begin{example}[Weil algebra]\label{Weil-example}
{\rm
 Let $\Fg$ be a (finite dimensional) Lie algebra and set $V = S(\Fg^*)$ - the polynomial algebra on  $\Fg$. Then $V$ is a right $\Fg$-module via the (co)adjoint action of $\Fg$. In other words,
\begin{equation}
 L_k:= \mathcal{L}_{X_k}.
\end{equation}
The role of $L^k$ is played by the multiplication of $\t^k$. That is
\begin{equation} L^k(\a)=\a\t^k.
\end{equation}
In this case, the equations \eqref{21} and \eqref{3} are satisfied and we obtain the Weil complex.
}\end{example}

\begin{example}[Truncated Weil algebra]
{\rm
 Let $V=S(\Fg^*)_{[2n]}$ be the truncated polynomial algebra on $\Fg$. With the same structure as it is defined in Example \ref{Weil-example} one obtains the differential complex $W(\Fg, S(\Fg^*)_{[2n]})$.
}\end{example}

To be able to interpret the coefficient space further, we introduce the crossed product  algebra

\begin{equation}
\widetilde{D}(\Fg) := S(\Fg^*) \al U(\Fg)
\end{equation}

In the next proposition, by  $\widetilde\Fg$ we mean $\Fg^\ast\al \Fg$ with the Lie bracket defined in  \eqref{tdg}.

\begin{proposition}
The algebras  $\widetilde{D}(\Fg)$ and $  U(\Fg^* \al \Fg)$ are isomorphic.
\end{proposition}

\begin{proof}
It is a simple  case of \cite[Theorem 7.2.3]{Maji}, that is
\begin{equation}
U(\Fg^* \al \Fg) = U(\Fg^*) \al U(\Fg) = S(\Fg^*) \al U(\Fg) = \widetilde{D}(\Fg)
\end{equation}
\end{proof}

Next, we recall the compatibility for a module over a crossed product algebra, for a proof see  \cite[Lemma 3.6]{RangSutl}.

\begin{lemma}
Let $\mathcal{H}$ be a Hopf algebra, and $A$ an $\mathcal{H}$-module algebra. Then $V$ is a right module on the crossed product algebra $A \al \mathcal{H}$ if and only if $V$ is a right module on $A$ and a right module on  $\mathcal{H}$ such that
\begin{align}\label{4}
(v \cdot h) \cdot a = (v \cdot (h\ps{1} \rhd a)) \cdot h\ps{2}
\end{align}
\end{lemma}

\begin{corollary}\label{module-cross}
Let $\Fg$ be a Lie algebra and $V$ be a vector space. Then, $V$ is a right module over $S(\Fg^*) \al U(\Fg)$ if and only if $V$ is a right module over $\Fg$, a right module over $S(\Fg^*)$ and \eqref{3} is satisfied.
\end{corollary}

We can now reformulate the Proposition \ref{22} as follows.

\begin{proposition}
The graded space  $(W^{\bullet}(\Fg,V),d_{\rm CE}+d_{\rm K})$ is a  complex if and only if $V$ is a unimodular stable right $\widetilde\Fg$-module.
\end{proposition}

\begin{example}[Weil algebra with generalized coefficients \cite{AlekMein}]
{\rm    Let $\mathcal{E}'(\Fg^*)$ be the convolution algebra of compactly supported distributions on $\Fg^*$.
The symmetric algebra $S(\Fg^*)$ is canonically  identified with the subalgebra of distributions supported at the origin. This immediately results with a natural $S(\Fg^*)$-module structure on $\mathcal{E}'(\Fg^*)$ via its own multiplication.

Regarding the coordinate functions $\mu_i$, $1 \leq i \leq N$ as multiplication operators, we also have $[\mu_i,\theta^j] = \delta_j^i$.

The Lie derivative is described as follows.
\begin{equation}\label{11}
\Lc_i = C^k_{ij}\theta^j\mu_k, \quad 1 \leq i \leq N.
\end{equation}
 Therefore,
\begin{align}
\tau: \Fg \to \End(\mathcal{E}'(\Fg^*)), \quad X_i \mapsto C^k_{ji}\mu_k\theta^j = -L_{X_i} - \delta(X_i)I
\end{align}
is a map of Lie algebras, and hence equips $\mathcal{E}'(\Fg^*)$ with a right $\Fg$-module structure.

We first observe that
\begin{equation}
\sum_i (v \cdot X_i) \lhd \theta^i = C^k_{ji}v\mu_k\theta^j\theta^i = 0,
\end{equation}
by the commutativity of $S(\Fg^*)$ and the anti-commutativity of the lower indices of the structure coefficients.

Secondly we observe
\begin{align}
\begin{split}
& (v \cdot X_i) \lhd \theta^t = C^k_{ji}v\mu_k\theta^j\theta^t = \\
& C^k_{ji}v\theta^t\mu_k\theta^j + C^t_{ji}v\theta^j = (v \lhd \theta^t) \cdot X_i + v \lhd (X_i \rhd \theta^t),
\end{split}
\end{align}
\ie $\mathcal{E}'(\Fg^*)$ is a right module over $S(\Fg^*) \al U(\Fg)$. Hence we have the complex $W(\Fg,\mathcal{E}'(\Fg^*))$.

One notices that  in \cite{AlekMein} the authors consider compact groups and their Lie algebras which are unimodular and hence $\d=0$. So, their  and our actions of $\Fg$ coincide.
}\end{example}

\subsection{Weyl algebra}

 Following \cite{Wall} Appendix 1, let $V$ be a (finite dimensional) vector space with dual $V^*$. Let $\mathscr{P}(V)$ be the algebra of all polynomials on $V$ and $S(V)$ the symmetric algebra on $V$. Let us use the notation $D(V)$ for the algebra of differential operators on $V$ with polynomial coefficients - the Weyl algebra on $V$. For any $v \in V$ we introduce the operator
\begin{equation}
\partial_v(f)(w) := \dt f(w + tv).
\end{equation}
As a result,  we get an injective algebra map $v \mapsto \partial_v \in D(V)$. As a differential  operator on $V$, $\partial_v$ is identified with the derivative with respect to $v^* \in V^*$.

Using the bijective linear  map  $\mathscr{P}(V) \ot S(V) \to D(V)$ defined as $f \ot v \mapsto fv$, and the fact that    $\mathscr{P}(V) \cong S(V^*)$, we conclude that  $D(V) \cong S(V^*) \ot S(V)$ as vector spaces.

Following  \cite{Dixm}, the standard representation of $D(V)$ is as follows. Let $\{{v_1}^*, \cdots ,{v_n}^*\}$ be a basis of $V^*$. Then, forming $E = \mathbb{C}[{v_1}^*, \cdots ,{v_n}^*]$, we consider the operators $P_i \in \End(E)$ as $\partial/\partial {v_i}^*$ and $Q^i \in \End(E)$ as multiplication by ${v_i}^*$. Then the relations are
\begin{align}
\begin{split}
& [P_i,Q^i] = I, \qquad [P_i,Q^j] = 0, \quad i \neq j \\
& [P_i,P_j] = 0, \qquad [Q^i,Q^j] = 0, \quad \forall~ i, j
\end{split}
\end{align}

It is observed that if $V$ is a module over $D(\Fg)$  then $(W^{\bullet}(\Fg,V),d_{\rm CE}+d_{\rm K})$ is a  complex \cite{KumaVerg}. We now briefly remark the relation of this result with our interpretation of the coefficient space \eqref{model-complex}. To this end, we first notice that if $V$ is a right module over the Weyl algebra $D(\Fg)$, then it is module over the Lie algebra $\Fg$ via the Lie algebra map

\begin{equation}
\tau:\Fg \to D(\Fg), \quad X_i \mapsto C^l_{ki}P_lQ^k.
\end{equation}

 Explicitly, we define the action of the Lie algebra as
\begin{equation}
v \cdot X_k = v\tau(X_k).
\end{equation}

On the other hand, $V$ is also a module over the symmetric algebra $S(\Fg^*)$ via

\begin{equation}
v \lhd \theta^k = vQ^k
\end{equation}

\begin{lemma}
Let $V$ be a right module over $D(\Fg)$. Then $V$ is unimodular stable.
\end{lemma}

\begin{proof}
We immediately observe that
\begin{equation}
\sum_i (v \cdot X_i) \lhd \theta^i = \sum_i v \cdot (\tau(X_i)Q^i) = \sum_{i,l,k}v \cdot (C^k_{li}P_kQ^lQ^i) = 0
\end{equation}
by the commutativity of $Q$s and the anti-commutativity of the lower indices of the structure coefficients.
\end{proof}

Next, to observe the condition \eqref{3}, we introduce the following map

\begin{equation}
\Phi:\widetilde{D}(\Fg) \to D(\Fg), \quad \theta^j \al X_i \mapsto Q^j\tau(X_i) = C^l_{ki}Q^jP_lQ^k
\end{equation}

\begin{lemma}
The map $\Phi:\widetilde{D}(\Fg) \to D(\Fg)$ is well-defined.
\end{lemma}

\begin{proof}
It is enough to prove $\Phi(X_i)\Phi(\theta^j) = \Phi(X_i\ps{1} \rhd \theta^j)\Phi(X_i\ps{2})$. To this, we observe
\begin{align}
\begin{split}
& RHS = \Phi(X_i\ps{1} \rhd \theta^j)\Phi(X_i\ps{2}) = -C^j_{ik}Q^k + C^l_{ki}Q^jP_lQ^k \\
& = C^l_{ki}P_lQ^jQ^k = C^l_{ki}P_lQ^kQ^j = \Phi(X_i)\Phi(\theta^j) = LHS.
\end{split}
\end{align}
\end{proof}

\begin{corollary}
If $V$ is a right module over $D(\Fg)$, then $V$ is a right module over $\widetilde{D}(\Fg) = S(\Fg^*) \al U(\Fg)$.
\end{corollary}

\section{Lie algebra homology and Poincar\'e duality}
In this section, for any Lie algebra $\Fg$ and any stable $\widetilde\Fg$-module  $V$  we define  a complex dual to the model complex and establish a Poincar\'e duality between these two complexes. The need for this new complex will be justified in the next sections.

\subsection{Lie algebra homology}
Let  $\Fg$ be a Lie algebra and $V$ be a right  $\Fg$-module.  We recall the Lie algebra homology complex $C_q(\Fg, V) = \wg^q \Fg \ot V$ by
\begin{equation}
\xymatrix{\cdots \ar[r]^{\p_{\rm CE}\;\;\;\;\;\;\;\;} & C_2(\Fg,V) \ar[r]^{\p_{\rm CE}\;\;\;} & C_1(\Fg,V) \ar[r]^{\;\;\;\;\;\;\;\p_{\rm CE}} & V}
\end{equation}
where
 \begin{align}
 \begin{split}
& \p_{\rm CE}(X_0 \wdots X_{q-1} \ot v ) = \sum_i (-1)^i  X_0 \wdots \widehat{X}_i \wdots
X_{q-1} \ot v \cdot X_i + \\
& \sum_{i<j} (-1)^{i+j} [X_i, X_j] \wg X_0 \wdots \widehat{X}_i \wdots
\widehat{X}_j \wdots X_{q-1} \ot v \\
\end{split}
\end{align}
We call the homology of the complex $(C_{\bullet}(\Fg, V),\p_{\rm CE})$ the Lie algebra homology of $\Fg$ with coefficients in $V$ and denote it by $H_{\bullet}(\Fg,V)$.

\subsection{Poincar\'e duality}

Let  $V$ to be a right $\Fg$-module and right $S(\Fg^*)$-module. We introduce the graded vector space
$C^n(\Fg,V):= \wedge^n \Fg\ot V$ with two differentials:
\begin{align}
\begin{split}
& \p_{\rm CE}:C^{n+1}(\Fg,V) \to C^n(\Fg,V) \\
& Y_0 \wdots Y_n \ot v \mapsto  \sum_{j} (-1)^j Y_0 \wdots \widehat{Y}_j \wdots Y_n \ot v \cdot Y_j \\
&+ \sum_{j,k} (-1)^{j+k} [Y_j,Y_k] \cdot Y_0 \wdots \widehat{Y}_j \wdots \widehat{Y}_k \wdots Y_n \ot v
\end{split}
\end{align}
which is the Lie algebra homology boundary and the second one  by
\begin{equation}
\partial_{\rm K}:C^n(\Fg,V) \to C^{n+1}(\Fg,V), \quad Y_1 \wdots Y_n \ot v \mapsto \sum_i X_i \wg Y_1 \wdots Y_n \ot v \lhd \theta^i
\end{equation}

We first justify that $\p_{\rm K}$ is a differential.

\begin{lemma}
We have $\p_{\rm K} \circ \p_{\rm K} = 0$.
\end{lemma}

\begin{proof}
We  observe that by the commutativity of $S(\Fg^*)$ and the anti-commutativity of the wedge product we have
\begin{align}
\begin{split}
& \p_{\rm K} \circ \p_{\rm K} (Y_1 \wdots Y_n \ot v) = \sum_i \p_{\rm K}(X_i \wg Y_1 \wdots Y_n \ot v \lhd \theta^i) \\
& = \sum_{i,j}X_j \wg X_i \wg Y_1 \wdots Y_n \ot v \lhd \theta^i\theta^j = 0.
\end{split}
\end{align}

\end{proof}

 We say that a right $\widetilde\Fg$-module $V$ is stable if
\begin{equation}\label{9}
\sum_i (v \lhd \theta^i) \cdot X_i = 0.
\end{equation}

\begin{proposition}\label{7}
The complex $(C^{\bullet}(\Fg,V), \p_{\rm CE}+\p_{\rm K})$ is a  complex if and only if $V$ is stable  right $\widetilde{\Fg}$-module.
\end{proposition}

\begin{proof}
First we observe that  $V$ is right $\widetilde\Fg$-module if and only if
\begin{equation}
(v \cdot X_k) \lhd \theta^t = v \lhd (X_k \rhd \theta^t) + (v \lhd \theta^t) \cdot X_k, \quad 1 \leq k,t \leq N.
\end{equation}
On the one hand we have
\begin{align}
\begin{split}
& \p_{\rm CE}(\p_{\rm K}(Y_0 \wdots Y_n \ot v)) = \sum_i \p_{\rm CE}(X_i \wg Y_0 \wdots Y_n \ot v \lhd \theta^i) = \\
& \sum_i Y_0 \wdots Y_n \ot (v \lhd \theta^i) \cdot X_i + \sum_{i,j} (-1)^{j+1} X_i \wg Y_0 \wdots \widehat{Y}_j \wdots Y_n \ot (v \lhd \theta^i) \cdot Y_j + \\
& \sum_{i,j} (-1)^{j+1} [X_i,Y_j] \wg Y_0 \wdots \widehat{Y}_j \wdots Y_n \ot v \lhd \theta^i + \\
& \sum_{i,j} (-1)^{j+k} [Y_j,Y_k]\wg X_i \wg Y_0 \wdots \widehat{Y}_j \wdots \widehat{Y}_k \wdots Y_n \ot v \lhd \theta^i,
\end{split}
\end{align}
and on the other hand
\begin{align}
\begin{split}
& \p_{\rm K}(\p_{\rm CE}(Y_0 \wdots Y_n \ot v)) = \sum_{j} (-1)^j \p_{\rm K}(Y_0 \wdots \widehat{Y}_j \wdots Y_n \ot v \cdot Y_j) + \\
& \sum_{j,k} (-1)^{j+k} \p_{\rm K}([Y_j,Y_k] \wg Y_0 \wdots \widehat{Y}_j \wdots \widehat{Y}_k \wdots Y_n \ot v) = \\
& \sum_{i,j} (-1)^j X_i \wg Y_0 \wdots \widehat{Y}_j \wdots Y_n \ot (v \wg Y_j) \lhd \theta^i + \\
& \sum_{i,j,k} (-1)^{j+k+1} [Y_j,Y_k] \wg X_i \wg Y_0 \wdots \widehat{Y}_j \wdots \widehat{Y}_k \wdots Y_n \ot v \lhd \theta^i.
\end{split}
\end{align}
Therefore, the complex is a mixed complex if and only if
\begin{align}\label{8}
\begin{split}
& (\p_{\rm CE} \circ \p_{\rm K} + \p_{\rm K} \circ \p_{\rm CE})(Y_0 \cdots Y_n \ot v) = \sum_i Y_0 \cdots Y_n \ot (v \lhd \theta^i) \cdot X_i + \\
& \sum_{i,j} (-1)^{j+1} X_i \cdot Y_0 \cdots \widehat{Y}_j \cdots Y_n \ot [(v \lhd \theta^i) \cdot Y_j - (v \cdot Y_j) \lhd \theta^i] + \\
& \sum_{i,j} (-1)^{j+1} [X_i,Y_j] \cdot Y_0 \cdots \widehat{Y}_j \cdots Y_n \ot v \lhd \theta^i = 0
\end{split}
\end{align}

Now, if we assume that  $(C^{\bullet}(\Fg,V),\p_{\rm CE}+\p_{\rm K} )$ is a complex, then firstly it is easy to see that the stability condition \eqref{9} is equivalent to   $\big(\p_{\rm CE}+\p_{\rm K}\big)^2(1 \ot v)=0.$

Secondly,   the equation \eqref{8} yields that $V$ is a  $\widetilde\Fg$-module;
\begin{equation}\label{10}
X_i \ot (v \lhd \theta^i) \cdot Y - X_i \ot (v \cdot Y) \lhd \theta^i + [X_i,Y] \ot v \lhd \theta^i = 0.
\end{equation}

The converse argument is obvious.

\end{proof}

Recall that the derivation $\delta: \Fg\ra \Cb$ is the trace of the adjoint representation of $\Fg$ on itself.

\begin{proposition}\label{12}
A vector space $V$ is a unimodular stable  right  $\widetilde\Fg$-module, if and only if  $V \ot \mathbb{C}_{\delta}$ is a  stable right $\widetilde\Fg$-module.
\end{proposition}
\begin{proof}

Indeed, if $V$ is unimodular stable right  $\widetilde\Fg$-module, that is $\sum_i (v \lhd X_i) \cdot \theta^i = 0$,  for any $v\in V$, then
\begin{align}
\begin{split}
& \sum_i ((v \ot 1_{\mathbb{C}}) \cdot \theta^i) \lhd X^i = \sum_i (v \cdot \theta^i) \cdot X_i \ot 1_{\mathbb{C}} + v\delta(X_i) \ot 1_{\mathbb{C}}  \\
& =\sum_i (v \cdot X_i) \lhd \theta^i \ot 1_{\mathbb{C}} = 0,
\end{split}
\end{align}
which proves that $V\ot \Cb_{\d}$ is stable. Similarly we observe that for $1 \leq i,j \leq N$,
\begin{align}
\begin{split}
& ((v \ot 1_{\mathbb{C}}) \cdot X_j) \lhd \theta^i = (v \cdot X_j \ot 1_{\mathbb{C}} + v\delta(X_j) \ot 1_{\mathbb{C}}) \lhd \theta^i = \\
& ((v \cdot X_j) \lhd \theta^i + v\delta(X_j) \lhd \theta^i) \ot 1_{\mathbb{C}} = \\
& (v \lhd (X_j \rhd \theta^i) + (v \lhd \theta^i) \cdot X_j + v\delta(X_j) \lhd \theta^i) \ot 1_{\mathbb{C}} = \\
& (v \ot 1_{\mathbb{C}}) \lhd (X_j \rhd \theta^i) + ((v \ot 1_{\mathbb{C}}) \lhd \theta^i) \cdot X_j
\end{split}
\end{align}
\ie $V \ot \mathbb{C}_{\delta}$ is a right  $\widetilde\Fg$-module.
The converse argument is similar.

\end{proof}

Let us  briefly recall the Poincar\'e isomorphism by
\begin{equation}
 \FD_P:\wedge^p \Fg^* \to \wedge^{n-p}\Fg, \qquad  \eta \mapsto \iota(\eta)\varpi,
\end{equation}
where $\varpi = X_1 \wedge \cdots \wedge X_n$ is the covolume element of $\Fg$.  By definition
  $\iota(\theta^i): \wedge^{\bullet} \Fg \to \wedge^{\bullet-1}\Fg$ is given  by
\begin{align}
\langle \iota(\theta^i)\xi, \theta^{j_1} \wedge \cdots \wedge \theta^{j_{r-1}} \rangle := \langle \xi, \theta^i \wedge \theta^{j_1} \wedge \cdots \wedge \theta^{j_{r-1}} \rangle, \quad  \xi \in \wedge^r\Fg.
\end{align}
Finally, for $\eta = \theta^{i_1} \wedge \cdots \wedge \theta^{i_k}$, the interior multiplication $\iota(\eta): \wedge^\bullet\Fg \to \wedge^{\bullet-p}\Fg $ is a derivation of degree $-p$ defined by
\begin{align}
\iota(\eta) := \iota(\theta^{i_k}) \circ \cdots \circ \iota(\theta^{i_1}).
\end{align}

\begin{proposition}
Let $V$ be a right module over  stable right  $\widetilde\Fg$-module. Then the Poincar\'e isomorphism induces a map of complexes between the complex $W(\Fg,V \ot \mathbb{C}_{-\delta})$ and the complex $C(\Fg,V)$.
\end{proposition}

\begin{proof}
Let us first introduce the notation $\widetilde{V}:=V \ot \mathbb{C}_{-\delta}$. We can identify $\widetilde{V}$ with $V$ as a vector space, but with the right $\Fg$-module structure deformed as $v \lhd X := v \cdot X - v\delta(X)$.

We prove the commutativity of the (co)boundaries via the (inverse) Poincar\'e isomorphism, \ie
\begin{align}
\begin{split}
& \FD_P^{-1}: \wedge^p \Fg \ot V \to \wedge^{N-p}\Fg^* \ot \widetilde{V} \\
& \xi \ot v \mapsto \FD_P^{-1}(\xi \ot v),
\end{split}
\end{align}
where for an arbitrary $\eta \in \wg^{N-p}\Fg$
\begin{align}
\langle \eta, \FD_P^{-1}(\xi \ot v) \rangle := \langle \eta\xi, \omega^* \rangle v.
\end{align}
Here, $\omega^* \in \wg^N\Fg^*$ is the volume form.

The commutativity of the diagram
$$
\xymatrix {
 \ar[d]_{\FD_P^{-1}} \wedge^p\Fg \ot V  \ar[r]^{\p_{\rm CE}} &  \wedge^{p-1}\Fg \ot V   \ar[d]^{\FD_P^{-1}} \\
 \wedge^{N-p}\Fg^* \ot \widetilde{V}   \ar[r]_{d_{\rm CE}} &  \wedge^{N-p+1}\Fg^* \ot \widetilde{V}
}
$$
follows from the Poincar\'e duality in Lie algebra homology - cohomology,
 \cite[Chapter VI, Section 3]{Knap}. For the commutativity of the diagram
$$
\xymatrix {
 \ar[d]_{\FD_P^{-1}} \wedge^p\Fg \ot V  \ar[r]^{\p_{\rm K}} &  \wedge^{p+1}\Fg \ot V   \ar[d]^{\FD_P^{-1}} \\
 \wedge^{N-p}\Fg^* \ot \widetilde{V}   \ar[r]_{d_{\rm K}} &  \wedge^{N-p-1}\Fg^* \ot \widetilde{V}
}
$$
we take an arbitrary $\xi \in \wg^p\Fg$, $\eta \in \wg^{N-p-1}\Fg$ and $v \in V$. Then
\begin{align}
\begin{split}
& \FD_P^{-1}(\p_{\rm K}(\xi \ot v))(\eta) = \langle \eta X_i \xi, \omega^* \rangle v \lhd \theta^i = \\
& (-1)^{N-p-1} \langle X_i \eta \xi, \omega^* \rangle v \lhd \theta^i = (-1)^{N-p-1}d_{\rm K}(\FD_P^{-1}(\xi \ot v))(\eta).
\end{split}
\end{align}
\end{proof}

\section{Lie algebra coaction and SAYD coefficients}

 In this section we identify the coefficients we discussed in the previous sections of this paper with stable-anti-Yetter-Drinfeld module over the universal enveloping algebra of the Lie algebra in question. To this end,   we  introduce the notion of comodule over a Lie algebra.

\subsection{SAYD modules and cyclic cohomology of Hopf algebras}

Let $\Hc$ be a Hopf algebra. By definition, a character $\d: \Hc\ra \Cb$ is an algebra map.
A group-like  $\s\in \Hc$ is the dual object of the character, \ie $\D(\s)=\s\ot \s$. The pair $(\d,\s)$ is called a modular pair in involution \cite{ConnMosc00} if
 \begin{equation}
 \d(\s)=1, \quad \text{and}\quad  S_\d^2=Ad_\s,
 \end{equation}
where  $Ad_\s(h)= \s h\s^{-1}$ and  $S_\d$ is defined by
\begin{equation}
S_\d(h)=\d(h\ps{1})S(h\ps{2}).
\end{equation}
We recall from \cite{HajaKhalRangSomm04-I} the definition of a  right-left  stable-anti-Yetter-Drinfeld module over a Hopf algebra $\Hc$. Let $V$ be a right module and left comodule over a Hopf algebra $\Hc$. We say that it is stable-anti-Yetter-Drinfeld (SAYD) module over $\Hc$ if
\begin{equation}
\Db(v\cdot h)= S(h\ps{3})v\ns{-1}h\ps{1}\ot v\ns{0}\cdot h\ps{2},\qquad  v\ns{0} \cdot v\ns{-1}=v,
\end{equation}
for any $v\in V$ and $h\in \Hc$.
It is shown in \cite{HajaKhalRangSomm04-I} that any MPI defines a one dimensional SAYD module and all one dimensional SAYD modules come this way.

Let $V$ be a right-left SAYD module over a Hopf algebra $\Hc$. Let
 \begin{equation}
 C^q(\Hc,V):= V\ot \Hc^{\ot q}, \quad q\ge 0.
 \end{equation}

We recall the following operators on $C^{\bullet}(\Hc,V)$
\begin{align*}
&\text{face operators} \quad\p_i: C^q(\Hc,V)\ra C^{q+1}(\Hc,V), && 0\le i\le q+1\\
&\text{degeneracy operators } \quad\s_j: C^q(\Hc,V)\ra C^{q-1}(\Hc,V),&& \quad 0\le j\le q-1\\
&\text{cyclic operators} \quad\tau: C^q(\Hc,V)\ra C^{q}(\Hc,V),&&
\end{align*}
by
\begin{align}
\begin{split}
&\p_0(v\ot h^1\odots h^q)=v\ot 1\ot h^1\odots h^q,\\
&\p_i(v\ot h^1\odots h^q)=v\ot h^1\odots h^i\ps{1}\ot h^i\ps{2}\odots h^q,\\
&\p_{q+1}(v\ot h^1\odots h^q)=v\ns{0}\ot h^1\odots h^q\ot v\ns{-1},\\
&\s_j (v\ot h^1\odots h^q)= (v\ot h^1\odots \ve(h^{j+1})\odots h^q),\\
&\tau(v\ot h^1\odots h^q)=v\ns{0}h^1\ps{1}\ot S(h^1\ps{2})\cdot(h^2\odots h^q\ot v\ns{-1}),
\end{split}
\end{align}
where $\Hc$ acts on $\Hc^{\ot q}$ diagonally.

The graded module $C(\Hc,V)$  endowed with the above operators is then a cocyclic module \cite{HajaKhalRangSomm04-II}, which means that $\p_i,$ $\s_j$ and $\tau$ satisfy the following identities
\begin{eqnarray}
\begin{split}
& \p_{j}  \p_{i} = \p_{i} \p_{j-1},   \hspace{35 pt}  \text{ if} \quad\quad i <j,\\
& \sigma_{j}  \sigma_{i} = \sigma_{i} \sigma_{j+1},    \hspace{30 pt}  \text{    if} \quad\quad i \leq j,\\
&\sigma_{j} \p_{i} =   \label{rel1}
 \begin{cases}
\p_{i} \sigma_{j-1},   \quad
 &\text{if} \hspace{18 pt}\quad\text{$i<j$}\\
\text{Id}   \quad\quad\quad
 &\text{if}\hspace{17 pt} \quad   \text{$i=j$ or $i=j+1$}\\
\p_{i-1} \sigma_{j}  \quad
 &\text{    if} \hspace{16 pt}\quad \text{$i>j+1$},\\
\end{cases}\\
&\tau\p_{i}=\p_{i-1} \tau, \hspace{43 pt} 1\le i\le  q\\
&\tau \p_{0} = \p_{q+1}, \hspace{43 pt} \tau \sigma_{i} = \sigma _{i-1} \tau,  \hspace{33 pt} 1\le i\le q\\ \label{rel2}
&\tau \sigma_{0} = \sigma_{n} \tau^2, \hspace{43 pt} \tau^{q+1} = \Id.
\end{split}
\end{eqnarray}

One uses the face operators to define the Hochschild coboundary
\begin{align}
\begin{split}
&b: C^{q}(\Hc,V)\ra C^{q+1}(\Hc,V), \qquad \text{by}\qquad b:=\sum_{i=0}^{q+1}(-1)^i\p_i
\end{split}
\end{align}
It is known that $b^2=0$. As a result, one obtains the Hochschild complex of the coalgebra $\Hc$ with coefficients in the bicomodule $V$. Here, the right comodule defined trivially. The cohomology of $(C^\bullet(\Hc,V),b)$ is denoted by $H^\bullet_{\rm coalg}(H,V)$.

One uses the rest of the operators to define the Connes boundary operator,
\begin{align}
\begin{split}
&B: C^{q}(\Hc,V)\ra C^{q-1}(\Hc,V), \qquad \text{by}\qquad B:=\left(\sum_{i=0}^{q}(-1)^{qi}\tau^{i}\right) \s_{q-1}\tau.
\end{split}
\end{align}

It is shown in \cite{Conn83} that for any cocyclic module we have $b^2=B^2=(b+B)^2=0$. As a result, one defines the cyclic cohomology of $\Hc$ with coefficients in   SAYD module $V$, which is  denoted by $HC^{\bullet}(\Hc,V)$, as the total  cohomology of the bicomplex
\begin{align}
C^{p,q}(\Hc,V)= \left\{ \begin{matrix} V\ot \Hc^{\ot q-p},&\quad\text{if}\quad 0\le p\le q, \\
&\\
0, & \text{otherwise.}
\end{matrix}\right.
\end{align}

One also defines the periodic cyclic cohomology of $\Hc$ with coefficients in $V$, which is denoted by $HP^\ast(\Hc,V)$, as the total cohomology of direct sum total  of the following bicomplex

\begin{align}
C^{p,q}(\Hc,V)= \left\{ \begin{matrix} V\ot \Hc^{\ot q-p},&\quad\text{if}\quad  p\le q, \\
&\\
0, & \text{otherwise.}
\end{matrix}\right.
\end{align}

It can be seen that the periodic cyclic complex and hence the cohomology is $\Zb_2$ graded.

\subsection{SAYD modules over Lie algebras}
  We need  to define the notion of comodule over a Lie algebra $\Fg$ to be able to make a passage from the stable $\widetilde\Fg$-modules we already defined in the previous sections to SAYD modules over the universal enveloping algebra $U(\Fg)$.

\begin{definition}
We say a vector space $V$ is a left comodule over the Lie algebra $\Fg$ if there is a map $\Db_{\Fg}: V \ra \Fg \ot V$ such that
\begin{equation}\label{g-comod}
v\nsb{-2}\wg v\nsb{-1}\ot v\nsb{0}=0,
\end{equation}
where $\Db_{\Fg}(v)=v\nsb{-1}\ot v\nsb{0}$, and $$v\nsb{-2}\ot v\nsb{-1}\ot v\nsb{0}= v\nsb{-1}\ot (v\nsb{0})\nsb{-1}\ot (v\nsb{0})\nsb{0}.$$
\end{definition}

\begin{proposition}\label{mod-comod}
Let $\Fg$ be a  Lie algebra and $V$ be  a vector space. Then, $V$ is a right $S(\Fg^*)$-module if and only if it is a left $\Fg$-comodule.
\end{proposition}

\begin{proof}
Assume  that $V$ is a right module over the symmetric algebra $S(\Fg^*)$. Then for any $v \in V$ there is an element $v\nsb{-1} \ot v\nsb{0} \in \Fg^{**} \ot V \cong \Fg \ot V$ such that for any $\theta \in \Fg^*$
\begin{equation}
v \lhd \theta = v\nsb{-1}(\theta)v\nsb{0} = \theta(v\nsb{-1})v\nsb{0}.
\end{equation}
Hence define the linear map $\Db_{\Fg}: V \ra \Fg \ot V$ by
\begin{equation}
 v \mapsto v\nsb{-1}\ot v\nsb{0}.
\end{equation}
The compatibility needed for $V$ to be a right module over $S(\Fg^*)$,which is  $(v \lhd \theta) \lhd \eta - (v \lhd \eta) \lhd \theta = 0$ translates directly into
\begin{equation}
\alpha(v\nsb{-2} \wedge v\nsb{-1}) \ot v\nsb{0} = (v\nsb{-2} \ot v\nsb{-1} - v\nsb{-1} \ot v\nsb{-2}) \ot v\nsb{0} = 0,
\end{equation}
where $\alpha:\wedge^2\Fg \to U(\Fg)^{\ot\, 2}$ is the anti-symmetrization map. Since the anti-symmetrization is injective, we have
\begin{equation}
v\nsb{-2} \wedge v\nsb{-1} \ot v\nsb{0} = 0.
\end{equation}
Hence, $V$ is a left $\Fg$-comodule.

Conversely, assume that $V$ is a left $\Fg$-comodule via the map $\Db_{\Fg}:V \to \Fg \ot V$ defined by $  v \mapsto v\nsb{-1} \ot v\nsb{0}$. We define the right action
\begin{equation}
V \ot S(\Fg^*) \to V, \quad v \ot \theta \mapsto v \lhd \theta := \theta(v\nsb{-1})v\nsb{0},
\end{equation}
for any $\theta \in \Fg^*$ and any $v \in V$. Thus,
\begin{equation}
(v \lhd \theta) \lhd \eta - (v \lhd \eta) \lhd \theta = (v\nsb{-2} \ot v\nsb{-1} - v\nsb{-1} \ot v\nsb{-2})(\theta \ot \eta) \ot v\nsb{0} = 0,
\end{equation}
proving that $V$ is a right module over $S(\Fg^*)$.
\end{proof}

Having understood the relation between the left $\Fg$-coaction and right $S(\Fg^*)$-action, it is natural to investigate the relation with left $U(\Fg)$-coaction.

Let $\Db:V \to U(\Fg) \ot V$ be a left $U(\Fg)$-comodule structure on the linear space $V$. Then composing via the canonical projection $\pi:U(\Fg) \to \Fg$, we get a linear map $\Db_{\Fg}:V \to \Fg \ot V$.
$$
\xymatrix {
\ar[dr]_{\Db_{\Fg}} V \ar[r]^{\Db} & U(\Fg) \ot V \ar[d]^{\pi \ot id} \\
& \Fg \ot V
}
$$

\begin{lemma}
If $\Db:V \to U(\Fg) \ot V$ is a coaction, then so is $\Db_{\Fg}:V \to \Fg \ot V$.
\end{lemma}

\begin{proof}
If we write $\Db(v) = v\snsb{-1} \ot v\snsb{0}$ then
\begin{align}
\begin{split}
& v\nsb{-2} \wedge v\nsb{-1} \ot v\nsb{0} = \pi(v\snsb{-2}) \wedge \pi(v\snsb{-1}) \ot v\snsb{0} = \\
& \pi(v\snsb{-1}\ps{1}) \wedge \pi(v\snsb{-1}\ps{2}) \ot v\snsb{0} = 0
\end{split}
\end{align}
by the cocommutativity of $U(\Fg)$.
\end{proof}

For the reverse process which is to obtain a $U(\Fg)$-comodule out of  a $\Fg$-comodule, we will need the following concept.
\begin{definition}\label{def locally nilpotent comodule}
Let $V$ be a $\Fg$-comodule via $\Db_{\Fg}:V \to \Fg \ot V$. Then we call the coaction locally conilpotent if it is conilpotent on any one dimensional subspace.
In other words, $\Db_{\Fg}:V \to \Fg \ot V$ is locally conilpotent if and only if for any $v \in V$ there exists $n \in \mathbb{N}$ such that $\Db^n_{\Fg}(v) = 0$.
\end{definition}

\begin{example}\rm{
If $V$ is an SAYD module on $U(\Fg)$, then by \cite[Lemma 6.2]{JaraStef} we have the filtration $V = \cup_{p \in \Zb}F_pV$ defined as $F_0V = V^{coU(\Fg)}$ and inductively
\begin{equation}
F_{p+1}V/F_pV = (V/F_pV)^{coU(\Fg)}
\end{equation}
Then the induced $\Fg$-comodule $V$ is locally conilpotent.
}\end{example}

\begin{example}\rm{
Let $\Fg$ be a Lie algebra and $S(\Fg^\ast)$ be the symmetric algebra on $\Fg^\ast$. For $V = S(\Fg^\ast)$, consider the coaction
\begin{equation}
S(\Fg^\ast) \to \Fg \ot S(\Fg^\ast),\quad \a \mapsto X_i \ot \a\t^i,
\end{equation}
called the Koszul coaction. The corresponding $S(\Fg^*)$-action on $V$ coincides with the multiplication of $S(\Fg^\ast)$. Therefore, the Koszul coaction is not locally conilpotent.

One notes that the Koszul coaction is locally conilpotent on any truncation of the symmetric algebra.
}\end{example}

Let  $\{U_k(\Fg)\}_{k \geq 0}$ be the canonical filtration of $U(\Fg)$, \ie
\begin{equation}
U_0(\Fg) = \Cb \cdot 1, \quad U_1(\Fg) = \Cb \cdot 1 \oplus \Fg, \quad U_p(\Fg) \cdot U_q(\Fg) \subseteq U_{p+q}(\Fg)
\end{equation}

Let us call an element in $U(\Fg)$ as symmetric homogeneous of degree $k$ if it is the canonical image of a symmetric homogeneous tensor of degree $k$ over $\Fg$. Let  $U^k(\Fg)$  be the set of all symmetric elements of degree $n$ in $U(\Fg)$.

We recall from  \cite[ Proposition 2.4.4]{Dixm} that

\begin{align}\label{aux-1}
U_k(\Fg) = U_{k-1}(\Fg) \oplus U^k(\Fg).
\end{align}
In other words, there is a (canonical) projection
\begin{align}
\begin{split}
& \theta_k:U_k(\Fg) \to U^k(\Fg) \cong U_k(\Fg)/U_{k-1}(\Fg) \\
& X_1 \cdots X_k \mapsto \sum_{\sigma \in S_k}X_{\sigma(1)} \cdots X_{\sigma(k)}.
\end{split}
\end{align}
So, fixing an ordered basis of the Lie algebra $\Fg$, we can say that the above map is bijective on the PBW-basis elements.

Let us  consider the unique derivation of $U(\Fg)$ extending the adjoint action of the Lie algebra $\Fg$ on itself,  and call it  $\ad(X):U(\Fg) \to U(\Fg)$ for any $X \in \Fg$.
By \cite[Proposition 2.4.9]{Dixm},  $\ad(X)(U^k(\Fg)) \subseteq U^k(\Fg)$ and $\ad(X)(U_k(\Fg)) \subseteq U_k(\Fg)$. So by applying  $\ad(X)$ to both sides of \eqref{aux-1},  we observe that  the preimage of $\ad(Y)(\sum_{\sigma \in S_k}X_{\sigma(1)} \cdots X_{\sigma(k)})$  is  $\ad(Y)(X_1 \cdots X_k)$.

\begin{proposition}\label{U(g)-coaction}
For a locally conilpotent $\Fg$-comodule $V$, the linear map
\begin{align}
\begin{split}
& \Db:V \to U(\Fg) \ot V \\
& v \mapsto 1 \ot v + \sum_{k \geq 1}\theta_k^{-1}(v\nsb{-k} \cdots v\nsb{-1}) \ot v\nsb{0}
\end{split}
\end{align}
defines a $U(\Fg)$-comodule structure.
\end{proposition}

\begin{proof}
For an arbitrary basis element $v^i \in V$, let us write
\begin{equation}\label{aux-54}
v^i\nsb{-1} \ot v^i\nsb{0} = \alpha^{ij}_k X_j \ot v^k
\end{equation}
where $\alpha^{ij}_k \in \mathbb{C}$. Then, by the coaction compatibility $v\nsb{-2} \wedge v\nsb{-1} \ot v\nsb{0} = 0$ we have
\begin{equation}
v^i\nsb{-2} \ot v^i\nsb{-1} \ot v^i\nsb{0} = \sum_{j_1 , j_2} \alpha^{ij_1j_2}_{l_2} X_{j_1} \ot X_{j_2} \ot v^{l_2},
\end{equation}
such that $\alpha^{ij_1j_2}_{l_2} := \alpha^{ij_1}_{l_1}\alpha^{l_1j_2}_{l_2}$ and $\alpha^{ij_1j_2}_{l_2} = \alpha^{ij_2j_1}_{l_2}$.

We have
\begin{equation}
\Db(v^i) = 1 \ot v^i + \sum_{k \geq 1}\sum_{j_1 \leq \cdots \leq j_k} \alpha^{ij_1 \cdots j_k}_{l_k} X_{j_1} \cdots X_{j_k} \ot v^{l_k},
\end{equation}

because  for $k \geq 1$
\begin{equation}
v^i\nsb{-k} \ot \cdots \ot v^i\nsb{-1} \ot v^i\nsb{0} = \sum_{j_1 , \cdots , j_k} \alpha^{ij_1 \cdots j_k}_{l_k} X_{j_1} \ot \cdots \ot X_{j_k} \ot v^{l_k},
\end{equation}
where $\alpha^{ij_1 \cdots j_k}_{l_k} := \alpha^{ij_1}_{l_1} \cdots \alpha^{l_{k-1}j_k}_{l_k}$,  and for any $\sigma \in S_k$ we have
\begin{equation}\label{aux-55}
\alpha^{ij_1 \cdots j_k}_{l_k} = \alpha^{ij_{\sigma(1)} \cdots j_{\sigma(k)}}_{l_k}.
\end{equation}

At this point, the counitality is immediate,
\begin{equation}
(\ve \ot id) \circ \Db (v^i) = v^i.
\end{equation}
On the other hand, to prove the coassociativity we first observe that
\begin{align}
\begin{split}
& (id \ot \Db) \circ \Db(v^i) = 1 \ot \Db(v^i) + \sum_{k \geq 1}\sum_{j_1 \leq \cdots \leq j_k} \alpha^{ij_1 \cdots j_k}_{l_k} X_{j_1} \cdots X_{j_k} \ot \Db(v^{l_k}) \\
& = 1 \ot 1 \ot v^i + \sum_{k \geq 1}\sum_{j_1 \leq \cdots \leq j_k} \alpha^{ij_1 \cdots j_k}_{l_k} 1 \ot X_{j_1} \cdots X_{j_k} \ot v^{l_k} + \\
& \sum_{k \geq 1}\sum_{j_1 \leq \cdots \leq j_k} \alpha^{ij_1 \cdots j_k}_{l_k} X_{j_1} \cdots X_{j_k} \ot 1 \ot v^{l_k} + \\
& \sum_{k \geq 1}\sum_{j_1 \leq \cdots \leq j_k} \alpha^{ij_1 \cdots j_k}_{l_k} X_{j_1} \cdots X_{j_k} \ot (\sum_{t \geq 1}\sum_{r_1 \leq \cdots \leq r_t} \alpha^{l_kr_1 \cdots r_t}_{s_t} X_{r_1} \cdots X_{r_t} \ot v^{s_t}),
\end{split}
\end{align}
where $\alpha^{l_kr_1 \cdots r_t}_{s_t} := \alpha^{l_kr_1}_{s_1} \cdots \alpha^{s_{t-1}r_t}_{s_t}$. Then we notice that
\begin{align}
\begin{split}
& \Delta(\sum_{k \geq 1}\sum_{j_1 \leq \cdots \leq j_k} \alpha^{ij_1 \cdots j_k}_{l_k} X_{j_1} \cdots X_{j_k}) \ot v^{l_k} = \sum_{k \geq 1}\sum_{j_1 \leq \cdots \leq j_k} \alpha^{ij_1 \cdots j_k}_{l_k} 1 \ot X_{j_1} \cdots X_{j_k} \ot v^{l_k} \\
& + \sum_{k \geq 1}\sum_{j_1 \leq \cdots \leq j_k} \alpha^{ij_1 \cdots j_k}_{l_k} X_{j_1} \cdots X_{j_k} \ot 1 \ot v^{l_k} \\
& + \sum_{k \geq 2}\sum_{j_1 \leq \cdots \leq r_1 \leq \cdots \leq r_p \leq \cdots \leq j_k} \alpha^{ij_1 \cdots j_k}_{l_k} X_{r_1} \cdots X_{r_p} \ot X_{j_1} \cdots \widehat{X}_{r_1} \cdots \widehat{X}_{r_p} \cdots X_{j_k} \ot v^{l_k} = \\
& \sum_{k \geq 1}\sum_{j_1 \leq \cdots \leq j_k} \alpha^{ij_1 \cdots j_k}_{l_k} 1 \ot X_{j_1} \cdots X_{j_k} \ot v^{l_k} + \\
& \sum_{k \geq 1}\sum_{j_1 \leq \cdots \leq j_k} \alpha^{ij_1 \cdots j_k}_{l_k} X_{j_1} \cdots X_{j_k} \ot 1 \ot v^{l_k} + \\
& \sum_{p \geq 1}\sum_{k-p \geq 1}\sum_{q_1 \leq \cdots \leq q_{k-p}}\sum_{r_1 \leq \cdots \leq r_p} \alpha^{ir_1 \cdots r_p}_{l_p}\alpha^{l_pq_1 \cdots q_{k-p}}_{l_k} X_{r_1} \cdots X_{r_p} \ot X_{q_1} \cdots X_{q_{k-p}} \ot v^{l_k},
\end{split}
\end{align}
where for the last equality we write the complement of $r_1 \leq \cdots \leq r_p$ in $j_1 \leq \cdots \leq j_k$ as $q_1 \leq \cdots \leq q_{k-p}$. Then \eqref{aux-55} implies that
\begin{equation}
\alpha^{ij_1 \cdots j_k}_{l_k} = \alpha^{ir_1 \cdots r_p q_1 \cdots q_{k-p}}_{l_k} = \alpha^{ir_1 \cdots r_p}_{l_p}\alpha^{l_pq_1 \cdots q_{k-p}}_{l_k}.
\end{equation}
As a result,
\begin{equation}
(id \ot \Db) \circ \Db(v^i) = (\Delta \ot id) \circ \Db(v^i).
\end{equation}
This is the coassociativity and the proof is now complete.
\end{proof}

Let us denote by $\rm  ^{\Fg}conil\Mc$ the subcategory of locally conilpotent left $\Fg$-comodules of the category of left $\Fg$-comodules $\, ^{\Fg}\Mc$ with colinear maps.

Assigning a $\Fg$-comodule $\Db_{\Fg}:V \to \Fg \ot V$ to a $U(\Fg)$-comodule $\Db:V \to U(\Fg) \ot V$ determines a functor
\begin{align}
\xymatrix{
^{U(\Fg)}\Mc  \ar@<.1 ex>[r]^{P} &  \, \rm  ^{\Fg}conil\Mc
}
\end{align}
Similarly, constructing a $U(\Fg)$-comodule from a $\Fg$-comodule determines a functor
\begin{align}
\xymatrix{
\rm  ^{\Fg}conil\Mc \ar@<.1 ex>[r]^{E} &  \, ^{U(\Fg)}\Mc
}
\end{align}

As a result, we can express the following proposition.

\begin{proposition}
The categories $\, \rm  ^{U(\Fg)}\Mc$ and $\, \rm  ^{\Fg}conil\Mc$ are isomorphic.
\end{proposition}

\begin{proof}
We show that the functors
$$
\begin{xy}
\xymatrix{
^{U(\Fg)}\Mc  \ar@<.6 ex>[r]^{P} & \ar@<.6 ex>[l]^{E}  \, \rm  ^{\Fg}conil\Mc
}
\end{xy}
$$
are inverses to each other.

If $\Db_{\Fg}:V \to \Fg \ot V$ is a locally conilpotent $\Fg$-comodule and $\Db:V \to U(\Fg) \ot V$ the corresponding $U(\Fg)$-comodule, by the very definition  the $\Fg$-comodule corresponding to $\Db:V \to U(\Fg) \ot V$ is exactly $\Db_{\Fg}:V \to \Fg \ot V$. This proves that
\begin{equation}
P \circ E\; =\; \Id_{\rm  ^{\Fg}conil\Mc}.
\end{equation}
Conversely, let us start with a $U(\Fg)$-comodule $\Db:V \to U(\Fg) \ot V$ and write the coaction by using the PBW-basis of $U(\Fg)$ as follows
\begin{equation}
v^i\ps{-1} \ot v^i\ps{0} = 1 \ot v^i + \sum_{k \geq 1}\sum_{j_1 \leq \cdots \leq j_k}\gamma^{ij_1 \cdots j_k}_{l_k} X_{j_1} \cdots X_{j_k} \ot v^{l_k}.
\end{equation}
So, the corresponding $\Fg$-comodule $\Db_{\Fg}:V \to \Fg \ot V$ is given as follows
\begin{equation}
v^i\nsb{-1} \ot v^i\nsb{0} = \pi(v^i\ps{-1}) \ot v^i\ps{0} = \sum_j \gamma^{ij}_k X_j \ot v^k.
\end{equation}
Finally, the $U(\Fg)$-coaction corresponding to this $\Fg$-coaction is defined on $v^i \in V$ as
\begin{equation}
v^i \mapsto 1 \ot v + \sum_{k \geq 1}\sum_{j_1 \leq \cdots \leq j_k}\gamma^{ij_1}_{l_1}\gamma^{l_1j_2}_{l_2} \cdots \gamma^{l_{k-1}j_k}_{l_k} X_{j_1} \cdots X_{j_k} \ot v^{l_k}
\end{equation}
Therefore, we can recover $U(\Fg)$-coaction we started with if and only if
\begin{equation}\label{aux-40}
\gamma^{ij_1 \cdots j_k}_{l_k} = \gamma^{ij_1}_{l_1}\gamma^{l_1j_2}_{l_2} \cdots \gamma^{l_{k-1}j_k}_{l_k}, \quad \forall k \geq 1
\end{equation}
The equation \eqref{aux-40} is a consequence of the coassociativity $\Db$. Indeed,  applying the coassociativity as
\begin{align}
(\Delta^{k-1} \ot id) \circ \Db = \Db^k
\end{align}
and comparing the coefficients of $X_{j_1} \ot \cdots \ot X_{j_k}$ we  conclude \eqref{aux-40} for any $k \geq 1$. Hence,  we proved
\begin{equation}
E \circ P = \Id_{^{U(\Fg)}\Mc}.
\end{equation}

The equation \eqref{aux-40} implies that if $\Db:V \to U(\Fg) \ot V$ is a left coaction, then its associated  $\Fg$-coaction  $\Db_{\Fg}:V \to \Fg \ot V$ is locally conilpotent.

\end{proof}

For a $\Fg$-coaction
\begin{align}
v \mapsto v\nsb{-1} \ot v\nsb{0}
\end{align}

the associated  $U(\Fg)$-coaction  is denoted by
\begin{align}
v \mapsto v\snsb{-1} \ot v\snsb{0}.
\end{align}

\begin{definition}
Let $V$ be a right module and left comodule over a Lie algebra $\Fg$. We call $V$ a right-left AYD over $\Fg$ if
\begin{align}
\Db_{\Fg}(v \cdot X) = v\nsb{-1} \ot v\nsb{0} \cdot X + [v\nsb{-1}, X] \ot v\nsb{0}.
\end{align}
Moreover, $V$ is called stable if
\begin{align}
v\nsb{0} \cdot v\nsb{-1} = 0.
\end{align}
\end{definition}

\begin{proposition}
Let $\Db_{\Fg}:V \to \Fg \ot V$ be a locally conilpotent $\Fg$-comodule and $\Db:V \to U(\Fg) \ot V$ the corresponding $U(\Fg)$-comodule structure. Then, $V$ is a right-left AYD over $\Fg$ if and only if it is a right-left AYD over $U(\Fg)$.
\end{proposition}

\begin{proof}
 Let us first assume $V$ to be a right-left AYD  module over $\Fg$.
For  $X \in \Fg$ and an element $v \in V$, AYD compatibility implies that
\begin{align}
\begin{split}
& (v \cdot X)\nsb{-k} \ot \cdots \ot (v \cdot X)\nsb{-1} \ot (v \cdot X)\nsb{0} = v\nsb{-k} \ot \cdots \ot v\nsb{-1} \ot v\nsb{0} \cdot X \\
& + [v\nsb{-k}, X] \ot \cdots \ot v\nsb{-1} \ot v\nsb{0} + v\nsb{-k} \ot \cdots \ot [v\nsb{-1}, X] \ot v\nsb{0}.
\end{split}
\end{align}
Multiplying in $U(\Fg)$, we get
\begin{align}
\begin{split}
& (v \cdot X)\nsb{-k} \cdots (v \cdot X)\nsb{-1} \ot (v \cdot X)\nsb{0}  = \\
& v\nsb{-k} \cdots v\nsb{-1} \ot v\nsb{0} \cdot X - ad(X)(v\nsb{-k} \cdots v\nsb{-1}) \ot v\nsb{0}.
\end{split}
\end{align}
So, for the extension $\Db:V \to U(\Fg) \ot V$ we have
\begin{align}
\begin{split}
& (v \cdot X)\snsb{-1} \ot (v \cdot X)\snsb{0} = 1 \ot v \cdot X + \sum_{k \geq 1}\theta_k^{-1}((v \cdot X)\nsb{-k} \cdots (v \cdot X)\nsb{-1}) \ot (v \cdot X)\nsb{0} \\
& = 1 \ot v \cdot X + \sum_{k \geq 1}\theta_k^{-1}(v\nsb{-k} \cdots v\nsb{-1}) \ot v\nsb{0} \cdot X - \sum_{k \geq 1}\theta_k^{-1}(ad(X)(v\nsb{-k} \cdots v\nsb{-1})) \ot v\nsb{0} \\
& = v\snsb{-1} \ot v\snsb{0} \cdot X - \sum_{k \geq 1}ad(X)(\theta_k^{-1}(v\nsb{-k} \cdots v\nsb{-1})) \ot v\nsb{0} \\
& = v\snsb{-1} \ot v\snsb{0} \cdot X - ad(X)(v\snsb{-1}) \ot v\snsb{0} = S(X\ps{3})v\snsb{-1}X\ps{1} \ot v\snsb{0} \cdot X\ps{2}.
\end{split}
\end{align}
Here on the third equality we used the fact that the operator $\ad$  commute with $\theta_k$, and on the fourth equality we used
\begin{align}
\begin{split}
& \sum_{k \geq 1}ad(X)(\theta_k^{-1}(v\nsb{-k} \cdots v\nsb{-1})) \ot v\nsb{0} = \\
& \sum_{k \geq 1}ad(X)(\theta_k^{-1}(v\nsb{-k} \cdots v\nsb{-1})) \ot v\nsb{0} + ad(X)(1) \ot v = ad(X)(v\snsb{-1}) \ot v\snsb{0}.
\end{split}
\end{align}
By  using the fact that  AYD condition is multiplicative,  we conclude that $\Db:M \to U(\Fg) \ot M$ satisfies the AYD condition on $U(\Fg)$.

Conversely assume that $V$ is a right-left AYD over $U(\Fg)$. We first observe that
\begin{align}
(\Delta \ot id) \circ \Delta (X) = X \ot 1 \ot 1 + 1 \ot X \ot 1 + 1 \ot 1 \ot X
\end{align}
Accordingly,
\begin{align}
\begin{split}
& \Db(v \cdot X) = v\snsb{-1}X \ot v\snsb{0} + v\snsb{-1} \ot v\snsb{0} \cdot X - Xv\snsb{-1} \ot v\snsb{0} \\
& = -ad(X)(v\snsb{-1}) \ot v\snsb{0} + v\snsb{-1} \ot v\snsb{0} \cdot X
\end{split}
\end{align}
It is known that  the projection map $\pi:U(\Fg) \to \Fg$ commutes with the adjoint representation. So
\begin{align}
\begin{split}
& \Db_{\Fg}(v \cdot X) = -\pi(ad(X)(v\snsb{-1})) \ot v\snsb{0} + \pi(v\snsb{-1}) \ot v\snsb{0} \cdot X \\
& = -ad(X)\pi(v\snsb{-1}) \ot v\snsb{0} + \pi(v\snsb{-1}) \ot v\snsb{0} \cdot X \\
& = [v\nsb{-1},X] \ot v\nsb{0} + v\nsb{-1} \ot v\nsb{0} \cdot X.
\end{split}
\end{align}
That  is,  $V$ is a right-left AYD over $\Fg$.
\end{proof}
\begin{lemma}
Let $\Db_{\Fg}:V \to \Fg \ot V$ be a locally conilpotent $\Fg$-comodule and $\Db:V \to U(\Fg) \ot V$ be the corresponding $U(\Fg)$-comodule structure. If $V$ is stable over $\Fg$, then it is stable over $U(\Fg)$.
\end{lemma}

\begin{proof}
Writing the $\Fg$-coaction in terms of basis elements as in \eqref{aux-54}, the stability reads
\begin{align}
v^i\nsb{0}v^i\nsb{-1} = \alpha^{ij}_kv^k \cdot X_j = 0, \quad \forall i
\end{align}
Therefore, for the corresponding $U(\Fg)$-coaction we have
\begin{align}
\begin{split}
& \sum_{j_1 \leq \cdots \leq j_k} \alpha^{ij_1}_{l_1} \cdots \alpha^{l_{k-1}j_k}_{l_k} v^{l_k} \cdot  (X_{j_1} \cdots X_{j_k}) = \\
& \sum_{j_1 \leq \cdots \leq j_{k-1}} \alpha^{ij_1}_{l_1} \cdots \alpha^{l_{k-2}j_{k-1}}_{l_{k-1}} (\sum_{j_k} \alpha^{l_{k-1}j_k}_{l_k} v^{l_k} \cdot X_{j_1}) \cdot (X_{j_2} \cdots X_{j_k}) = \\
& \sum_{j_2, \cdots , j_k} \alpha^{ij_k}_{l_1} \cdots \alpha^{l_{k-2}j_{k-1}}_{l_{k-1}} (\sum_{j_1} \alpha^{l_{k-1}j_1}_{l_k} v^{l_k} \cdot X_{j_1}) \cdot (X_{j_2} \cdots X_{j_k}),
\end{split}
\end{align}
where on the second equality we used \eqref{aux-55}. This  immediately implies that
\begin{align}
v^i\snsb{0} \cdot v^i\snsb{-1} = v^i.
\end{align}
That is, the stability over $U(\Fg)$.
\end{proof}

However, the converse is not true.
\begin{example}{\rm
It is known that $U(\Fg)$, as a left $U(\Fg)$-comodule via $\Delta:U(\Fg) \to U(\Fg) \ot U(\Fg)$ and a right $\Fg$-module via $ad:U(\Fg) \ot \Fg \to U(\Fg)$ is stable. However, the associated  $\Fg$-comodule, is  no longer  stable. Indeed, for $u = X_1X_2X_3 \in U(\Fg)$, we have
\begin{align}
u\nsb{-1} \ot u\nsb{0} = X_1 \ot X_2X_3 + X_2 \ot X_1X_3 + X_3 \ot X_1X_2
\end{align}
Then,
\begin{align}
u\nsb{0} \cdot u\nsb{-1} = [[X_1,X_2],X_3] + [[X_2,X_1],X_3] + [[X_1,X_3],X_2] = [[X_1,X_3],X_2]
\end{align}
which is not necessarily zero.
}\end{example}

The following is the main result of this section.
\begin{proposition}\label{23}
Let $V$ be a vector space, and $\Fg$ be a Lie algebra. Then, $V$ is a stable right $\widetilde\Fg$-module if and only if it is a right-left SAYD module over $\Fg$.
\end{proposition}

\begin{proof}
Let us first assume that $V$ is a stable right $\widetilde\Fg$-module. Since  $V$ is a right $S(\Fg^*)$-module it is  a left $\Fg$-comodule by Proposition \ref{mod-comod}. Accordingly
\begin{align}
\begin{split}
& [v\nsb{-1}, X_j] \ot v\nsb{0} + v\nsb{-1} \ot v\nsb{0} \cdot X_j = \\
& [X_l,X_j]\theta^l(v\nsb{-1}) \ot v\nsb{0} + X_t\theta^t(v\nsb{-1}) \ot v\nsb{0} \cdot X_j = \\
& X_tC^t_{lj}\theta^l(v\nsb{-1}) \ot v\nsb{0} + X_t\theta^t(v\nsb{-1}) \ot v\nsb{0} \cdot X_j = \\
& X_t \ot [v \lhd (X_j \rhd \theta^t) + (v \lhd \theta^t) \cdot X_j] = \\
& X_t \ot (v \cdot X_j) \lhd \theta^t = X_t\theta^t((v \cdot X_j)\nsb{-1}) \ot (v \cdot X_j)\nsb{0} = \\
& (v \cdot X_j)\nsb{-1} \ot (v \cdot X_j)\nsb{0}
\end{split}
\end{align}
This proves that $V$ is a right-left AYD module over $\Fg$. On the other hand, for any $v \in V$,
\begin{equation}
v\nsb{0} \cdot v\nsb{-1} = \sum_i v\nsb{0} \cdot X_i\theta^i(v\nsb{-1}) = \sum_i (v \lhd \theta^i) \cdot X_i = 0
\end{equation}
Hence, $V$ is stable too. As a result, $V$ is SAYD over $\Fg$.

Conversely, assume that  $V$ is a  right-left SAYD module over $\Fg$. So $V$  is  a right module over $S(\Fg^*)$ and a right module over $\Fg$. In addition we see that
\begin{align}
\begin{split}
& v \lhd (X_j \rhd \theta^i) + (v \lhd \theta^i) \cdot X_j = C^i_{kj}v \lhd \theta^k + (v \lhd \theta^i) \cdot X_j = \\
& C^i_{kj}\theta^k(v\nsb{-1})v\nsb{0} + \theta^i(v\nsb{-1})v\nsb{0} \cdot X_j = \\
& \theta^i([v\nsb{-1},X_j])v\nsb{0} + \theta^i(v\nsb{-1})v\nsb{0} \cdot X_j = \\
& (\theta^i \ot id)([v\nsb{-1}, X_j] \ot v\nsb{0} + v\nsb{-1} \ot v\nsb{0} \cdot X_j) = \\
& \theta^t((v \cdot X_j)\nsb{-1}) (v \cdot X_j)\nsb{0} = (v \cdot X_j) \lhd \theta^i
\end{split}
\end{align}
Thus, $V$ is a right $\widetilde\Fg$-module by equation \eqref{3}. Finally, we prove  the stability by
\begin{equation}
\sum_i (v \lhd \theta^i) \cdot X_i = \sum_i v\nsb{0} \cdot X_i\theta^i(v\nsb{-1}) = v\nsb{0} \cdot v\nsb{-1} = 0.
\end{equation}
\end{proof}

\begin{corollary}
Any right module over the Weyl algebra $D(\Fg)$ is a right-left SAYD module over the Lie algebra $\Fg$.
\end{corollary}

Finally, we state an analogous of Lemma 2.3 \cite{HajaKhalRangSomm04-I} to show that the category of $^\Fg\mathcal{AYD}_\Fg$ is monoidal.

\begin{proposition}
Let $M$ and $N$ be two right-left AYD modules over $\Fg$. Then $M \ot N$ is also a right-left AYD over $\Fg$ via
the coaction
\begin{equation}
\Db_{\Fg}:M \ot N \to \Fg \ot M \ot N, \quad m \ot n \mapsto m\nsb{-1} \ot m\nsb{0} \ot n + n\nsb{-1} \ot m \ot n\nsb{0}
\end{equation}
and the action
\begin{equation}
M \ot N \ot \Fg \to M \ot N, \quad (m \ot n) \cdot X = m \cdot X \ot n + m \ot n \cdot X
\end{equation}
\end{proposition}

\begin{proof}

We simply verify that
\begin{align}
\begin{split}
& [(m \ot n)\nsb{-1},X] \ot (m \ot n)\nsb{0} + (m \ot n)\nsb{-1} \ot (m \ot n)\nsb{0} \cdot X = \\
& [m\nsb{-1},X] \ot m\nsb{0} \ot n + [n\nsb{-1},X] \ot m \ot n\nsb{0} + \\
& m\nsb{-1} \ot (m\nsb{0} \ot n) \cdot X + n\nsb{-1} \ot (m \ot n\nsb{0}) \cdot X = \\
& (m \cdot X)\nsb{-1} \ot (m \cdot X)\nsb{0} \ot n + n\nsb{-1} \ot m \cdot X \ot n\nsb{0} + \\
& m\nsb{-1} \ot m\nsb{0} \ot n\cdot X + (n\cdot X)\nsb{-1} \ot m \ot (n \cdot X)\nsb{0} = \\
& \Db_{\Fg}(m \cdot X \ot n + m \ot n \cdot X) = \Db_{\Fg}((m \ot n) \cdot X).
\end{split}
\end{align}
\end{proof}

\subsection{Examples}

This subsection is devoted to examples to illustrate the notion of SAYD module over a Lie algebra. We consider the representations and corepresentations of a Lie algebra $\Fg$ on a finite dimensional vector space $V$ in terms of matrices. We then investigate the SAYD condition as a relation between these matrices and the Lie algebra structure of $\Fg$.

 Let also $V$ be a $n$ dimensional $\Fg$-module with a basis $\{v^1, \cdots, v^n\}$. We  express the module structure as
\begin{equation}
m^i \cdot X_j = \beta^i_{jk}m^k, \quad \beta^i_{jk} \in \Cb.
\end{equation}
In this way, for any basis element $X_j \in \Fg$ we obtain a matrix $B_j \in M_n(\mathbb{C})$ such that
\begin{equation}
(B_j)^i_k := \beta^i_{jk}.
\end{equation}
Let $\Db_{\Fg}:V \to \Fg \ot V$ be a coaction. We write  the coaction as
\begin{equation}
\Db_{\Fg}(v^i) = \alpha^{ij}_k X_j \ot v^k, \quad \alpha^{ij}_k\in \Cb.
\end{equation}
This way  we get a matrix $A^j \in M_n(\mathbb{C})$ for any basis element $X_j \in \Fg$ such that
\begin{equation}
(A^j)^i_k := \alpha^{ij}_k.
\end{equation}

\begin{lemma}
Linear map $\Db_{\Fg}:M \to \Fg \ot M$ forms a right $\Fg$-comodule if and only if
\begin{equation}
A^{j_1} \cdot A^{j_2} = A^{j_2} \cdot A^{j_1}.
\end{equation}
\end{lemma}

\begin{proof}
It is just the translation of the coaction compatibility $v^i\nsb{-2} \wedge v^i\nsb{-1} \ot v^i\nsb{0} = 0$  in terms of the matrices $A^i$.
\end{proof}

\begin{lemma}
Right $\Fg$-module left $\Fg$-comodule $V$ is stable if and only if
\begin{equation}\label{aux-4}
\sum_j A^j \cdot B_j = 0.
\end{equation}
\end{lemma}

\begin{proof}
By the definition of the stability,
\begin{equation}
v^i\nsb{0} \cdot v^i\nsb{-1} = \alpha^{ij}_{k} v^k \cdot X_j = \alpha^{ij}_{k}\beta^k_{jl} v^l = 0
\end{equation}
Therefore,
\begin{equation}
\alpha^{ij}_{k}\beta^k_{jl} = (A^j)^i_k (B_j)^k_l = (A^j \cdot B_j)^i_l = 0.
\end{equation}
\end{proof}

We proceed to express the AYD condition.

\begin{lemma}
The $\Fg$-module-comodule $V$ is a right-left AYD  if and only if
\begin{equation}\label{aux-2}
[B_q, A^j] = \sum_s A^sC^j_{sq}.
\end{equation}
\end{lemma}

\begin{proof}
We first observe
\begin{align}
\begin{split}
& \Db_{\Fg}(v^p \cdot X_q) = \Db_{\Fg}(\beta^p_{qk}v^k) = \beta^p_{qk}\alpha^{kj}_l X_j \ot v^l \\
& = (B_q)^p_k(A^j)^k_l X_j \ot v^l = (B_q \cdot A^j)^p_l X_j \ot v^l.
\end{split}
\end{align}
On the other hand, writing $\Db_{\Fg}(v^p) = \alpha^{pj}_{l} X_j \ot v^l$,
\begin{align}
\begin{split}
& [v^p\nsb{-1},X_q] \ot v^p\nsb{0} + v^p\nsb{-1} \ot v^p\nsb{0} \cdot X_q = \alpha^{ps}_{l}[X_s,X_q] \ot v^l + \alpha^{pj}_{t} X_j \ot v^t \cdot X_q \\
& = \alpha^{ps}_{l}C^j_{sq}X_j \ot v^l + \alpha^{pj}_{t}\beta^t_{ql} X_j \ot v^l = (\alpha^{ps}_{l}C^j_{sq} + (A^j \cdot B_q)^p_l)X_j \ot v^l.
\end{split}
\end{align}
\end{proof}

\begin{remark}\rm{
The stability and the AYD conditions are independent of the choice of basis. Let $\{Y_j\}$ be another basis with
\begin{equation}
Y_j = \gamma^l_jX_l, \qquad X_j = (\gamma^{-1})^l_jY_l.
\end{equation}
Hence, the action and coaction matrices are
\begin{equation}
\widetilde{B}_q = \gamma^l_qB_l, \quad \widetilde{A}^j = A^l(\gamma^{-1})^j_l,
\end{equation}
respectively. Therefore,
\begin{equation}
\sum_j \widetilde{A}^j \cdot \widetilde{B}_j = \sum_{j,l,s} A^l(\gamma^{-1})^j_l\gamma^s_jB_s = \sum_{l,s} A^lB_s\delta^l_s = \sum_j A^j \cdot B_j = 0,
\end{equation}
proving that the stability is independent of the choice of basis. Secondly, we have
\begin{equation}
[\widetilde{B}_q,\widetilde{A}^j] = \gamma^s_q(\gamma^{-1})^j_r[B_s,A^r] = \gamma^s_q(\gamma^{-1})^j_rA^lC^r_{ls}.
\end{equation}
If we write $[Y_p,Y_q] = \widetilde{C}^r_{pq}Y_r$, then it is immediate to observe that
\begin{equation}
\gamma^s_qC^r_{ls}(\gamma^{-1})^j_r = (\gamma^{-1})^s_l\widetilde{C}^j_{sq}.
\end{equation}
Therefore,
\begin{equation}
[\widetilde{B}_q,\widetilde{A}^j] = A^l\gamma^s_qC^r_{ls}(\gamma^{-1})^j_r = A^l(\gamma^{-1})^s_l\widetilde{C}^j_{sq} = \widetilde{A}^s\widetilde{C}^j_{sq}.
\end{equation}
This observation proves that the AYD condition is independent of the choice of basis.
}\end{remark}

Next, considering the Lie algebra $\Fs\Fl(2)$, we determine the SAYD modules over simple $\Fs\Fl(2)$-modules. First of all, we fix a basis of $\Fs\Fl(2)$ as follows.
\begin{equation}
X_1 = \left(
        \begin{array}{cc}
          0 & 1 \\
          0 & 0 \\
        \end{array}
      \right), \qquad X_2 = \left(
        \begin{array}{cc}
          0 & 0 \\
          1 & 0 \\
        \end{array}
      \right), \qquad X_3 = \left(
        \begin{array}{cc}
          1 & 0 \\
          0 & -1 \\
        \end{array}
      \right).
\end{equation}

\begin{example}\rm{
Let $V = <\{v^1,v^2\}>$ be a two dimensional simple $\Fs\Fl(2)$-module. Then, by  \cite{Kass}, the representation
\begin{align}
\rho:\Fs\Fl(2) \to \Fg\Fl(V)
\end{align}
is the inclusion $\rho:\Fs\Fl(2) \hookrightarrow \Fg\Fl(2)$. Therefore, we have
\begin{equation}
B_1 = \left(
        \begin{array}{cc}
          0 & 0 \\
          1 & 0 \\
        \end{array}
      \right), \quad B_2 = \left(
        \begin{array}{cc}
          0 & 1 \\
          0 & 0 \\
        \end{array}
      \right), \quad B_3 = \left(
        \begin{array}{cc}
          1 & 0 \\
          0 & -1 \\
        \end{array}
      \right).
\end{equation}

We want to find
\begin{align}
A^1 = \left(
        \begin{array}{cc}
          x^1_1 & x^1_2 \\
          x^2_1 & x^2_2 \\
        \end{array}
      \right), \qquad A^2 = \left(
        \begin{array}{cc}
          y^1_1 & y^1_2 \\
          y^2_1 & y^2_2 \\
        \end{array}
      \right), \qquad A^3 = \left(
        \begin{array}{cc}
          z^1_1 & z^1_2 \\
          z^2_1 & z^2_2 \\
        \end{array}
      \right),
\end{align}
such that together with the $\Fg$-coaction $\Db_{\Fs\Fl(2)}:V \to \Fs\Fl(2) \ot V$, defined as $v^i \mapsto (A^j)^i_k X_j \ot v^k$, $V$ becomes a right-left SAYD over $\Fs\Fl(2)$. We first express the stability condition. To this end,
\begin{align}
A^1 \cdot B_1 = \left(
        \begin{array}{cc}
          x^1_2 & 0 \\
          x^2_2 & 0 \\
        \end{array}
      \right), \qquad A^2 \cdot B_2 = \left(
        \begin{array}{cc}
          0 & y^1_1 \\
          0 & y^2_1 \\
        \end{array}
      \right), \qquad A^3 \cdot B_3 = \left(
        \begin{array}{cc}
          z^1_1 & -z^1_2 \\
          z^2_1 & -z^2_2 \\
        \end{array}
      \right),
\end{align}
and hence, the stability is
\begin{align}
\sum_j A^j \cdot B_j = \left(
                         \begin{array}{cc}
                           x^1_2 + z^1_1 & y^1_1 - z^1_2 \\
                           x^2_2 + z^2_1 & y^2_1 - z^2_2 \\
                         \end{array}
                       \right) = 0.
\end{align}
Next, we consider the AYD condition
\begin{equation}
[B_q, A^j] = \sum_s A^sC^j_{sq}.
\end{equation}
For $j = 1 = q$,
\begin{equation}
A^1 = \left(
        \begin{array}{cc}
          x^1_1 & 0 \\
          x^2_1 & x^2_2 \\
        \end{array}
      \right), \qquad A^2 = \left(
        \begin{array}{cc}
          0 & y^1_2 \\
          0 & y^2_2 \\
        \end{array}
      \right), \qquad A^3 = \left(
        \begin{array}{cc}
          0 & 0 \\
          z^2_1 & 0 \\
        \end{array}
      \right).
\end{equation}
Similarly, for $q = 2$ and $j = 1$, we arrive
\begin{equation}
A^1 = \left(
        \begin{array}{cc}
          0 & 0 \\
          0 & 0 \\
        \end{array}
      \right), \qquad A^2 = \left(
        \begin{array}{cc}
          0 & y^1_2 \\
          0 & y^2_2 \\
        \end{array}
      \right), \qquad A^3 = \left(
        \begin{array}{cc}
          0 & 0 \\
          0 & 0 \\
        \end{array}
      \right).
\end{equation}
Finally, for $j = 1$ and $q=2$ we conclude
\begin{equation}
A^1 = \left(
        \begin{array}{cc}
          0 & 0 \\
          0 & 0 \\
        \end{array}
      \right), \qquad A^2 = \left(
        \begin{array}{cc}
          0 & 0 \\
          0 & 0 \\
        \end{array}
      \right), \qquad A^3 = \left(
        \begin{array}{cc}
          0 & 0 \\
          0 & 0 \\
        \end{array}
      \right).
\end{equation}
Thus, the only $\Fs\Fl(2)$-comodule structure that makes a 2-dimensional simple $\Fs\Fl(2)$-module $V$ to be a right-left SAYD over $\Fs\Fl(2)$ is the trivial comodule structure.
}\end{example}

\begin{example}\label{aux-62}
{\rm
We investigate all possible coactions that make the truncated symmetric algebra $S(\Fs\Fl(2)^*)_{[2]}$  an SAYD module over  $\Fs\Fl(2)$.

A vector space basis  of $S(\Fs\Fl(2)^*)_{[2]}$ is $\{1 =\theta^0, \theta^1, \theta^2, \theta^3\}$ and the Kozsul coaction is
\begin{align}
\begin{split}
& S(\Fs\Fl(2)^*)_{[2]} \to \Fs\Fl(2) \ot S(\Fs\Fl(2)^*)_{[2]} \\
& \theta^0 \mapsto X_1 \ot \theta^1 + X_2 \ot \theta^2 + X_3 \ot \theta^3 \\
& \theta^i \mapsto 0 ,\qquad i = 1,2,3
\end{split}
\end{align}
We first determine the right $\Fs\Fl(2)$ action to find the matrices $B_1, B_2, B_3$.  We have
\begin{align}
\theta^i \lhd X_j (X_q) = \theta^i \cdot X_j (X_q) = \theta^i([X_j, X_q]).
\end{align}
Therefore,
\begin{align}
B_1 = \left(
                    \begin{array}{cccc}
                      0 & 0 & 0 & 0 \\
                      0 & 0 & 0 & -2 \\
                      0 & 0 & 0 & 0 \\
                      0 & 0 & 1 & 0 \\
                    \end{array}
                  \right), \quad B_2 = \left(
                    \begin{array}{cccc}
                      0 & 0 & 0 & 0 \\
                      0 & 0 & 0 & 0 \\
                      0 & 0 & 0 & 2 \\
                      0 & -1 & 0 & 0 \\
                    \end{array}
                  \right), \quad B_3 = \left(
                    \begin{array}{cccc}
                      0 & 0 & 0 & 0 \\
                      0 & 2 & 0 & 0 \\
                      0 & 0 & -2 & 0 \\
                      0 & 0 & 0 & 0 \\
                    \end{array}
                  \right)
\end{align}
Let  $A^1 = (x^i_k), A^2 = (y^i_k), A^3 = (z^i_k)$ represent the $\Fg$-coaction on $V$. According to the above expression of $B_1, B_2, B_3$, the stability is
\begin{align}
\sum_j A^j \cdot B_j = \left(
                    \begin{array}{cccc}
                      0 & y^0_3 + 2z^0_1 & x^0_3 - 2z^0_2 & -2x^0_1 + 2y^0_2 \\
                      0 & y^1_3 + 2z^1_1 & x^1_3 - 2z^1_2 & -2x^1_1 + 2y^1_2 \\
                      0 & y^2_3 + 2z^2_1 & x^2_3 - 2z^2_2 & -2x^2_1 + 2y^2_2 \\
                      0 & y^3_3 + 2z^3_1 & x^3_3 - 2z^3_2 & -2x^3_1 + 2y^3_2 \\
                    \end{array}
                  \right) = 0.
\end{align}
As before, we make the following observations. First,
\begin{align}
[B_1, A^1] = \left(
                    \begin{array}{cccc}
                      0 & 0 & -x^0_3 & 2x^0_1 \\
                      -2x^3_0 & -2x^3_1 & -2x^3_2 -x^1_3 & -2x^3_3 + 2x^1_1 \\
                      0 & 0 & -x^2_3 & 2x^2_1 \\
                      x^2_0 & x^2_1 & x^2_2 - x^3_3 & x^2_3 - 2x^3_1 \\
                    \end{array}
                  \right) = 2A^3
\end{align}
and next
\begin{align}
[B_2, A^1] = \left(
                    \begin{array}{cccc}
                      0 & x^0_3 & 0 & -2x^0_2 \\
                      0 & x^1_3 & 0 & -2x^1_2 \\
                      2x^3_0 & 2x^3_1 + x^2_3 & 2x^3_2 & 2x^3_3 - 2x^2_2 \\
                      -x^1_0 & -x^1_1 + x^3_3 & -x^1_2 & -x^1_3 - 2x^3_2 \\
                    \end{array}
                  \right) = 0
\end{align}
Finally,
\begin{align}
[B_3, A^1] = \left(
                    \begin{array}{cccc}
                      0 & -2x^0_1 & 0 & 0 \\
                      0 & 0 & 0 & 0 \\
                      -2x^2_0 & -4x^2_1 & 0 & -2x^2_3 \\
                      0 & -2x^3_1 & 0 & 0 \\
                    \end{array}
                  \right) = -2A^1.
\end{align}
Hence, together with the stability one gets
\begin{align}
A^1 = \left(
                    \begin{array}{cccc}
                      0 & x^0_1 & 0 & 0 \\
                      0 & 0 & 0 & 0 \\
                      x^2_0 & 0 & 0 & 0 \\
                      0 & 0 & 0 & 0 \\
                    \end{array}
                  \right)
\end{align}
and
\begin{align}
[B_1, A^1] = \left(
                    \begin{array}{cccc}
                      0 & 0 & 0 & 2x^0_1 \\
                      0 & 0 & 0 & 0 \\
                      0 & 0 & 0 & 0 \\
                      x^2_0 & 0 & 0 & 0 \\
                    \end{array}
                  \right) = 2A^3.
\end{align}
Similarly one computes
\begin{align}
[B_1, A^2] = \left(
                    \begin{array}{cccc}
                      0 & 0 & 0 & 2y^0_1 \\
                      -2y^3_0 & -2y^3_1 & 0 & 2y^1_1 \\
                      0 & 0 & 0 & 2y^2_1 \\
                      y^2_0 & y^2_1 & 0 & 2y^3_1 \\
                    \end{array}
                  \right) = 0,
\end{align}
as well as
\begin{align}
[B_2, A^2] = \left(
                    \begin{array}{cccc}
                      0 & 0 & 0 & -2y^0_2 \\
                      0 & 0 & 0 & 0 \\
                      0 & 0 & 0 & 0 \\
                      -y^1_0 &  & 0 & 0 \\
                    \end{array}
                  \right) = -2A^3,
\end{align}
and $[B_3, A^2] = 2A^2$. We conclude that
\begin{align}
A^1 = \left(
                    \begin{array}{cccc}
                      0 & c & 0 & 0 \\
                      0 & 0 & 0 & 0 \\
                      d & 0 & 0 & 0 \\
                      0 & 0 & 0 & 0 \\
                    \end{array}
                  \right), \mbox{  } A^2 = \left(
                    \begin{array}{cccc}
                      0 & 0 & c & 0 \\
                      d & 0 & 0 & 0 \\
                      0 & 0 & 0 & 0 \\
                      0 & 0 & 0 & 0 \\
                    \end{array}
                  \right), \mbox{  } A^3 = \left(
                    \begin{array}{cccc}
                      0 & 0 & 0 & c \\
                      0 & 0 & 0 & 0 \\
                      0 & 0 & 0 & 0 \\
                      \frac{1}{2} d & 0 & 0 & 0 \\
                    \end{array}
                  \right)
\end{align}
One notes that for $c = 1, d = 0$ one recovers  the Kozsul coaction, but obviously it is not the only choice.
}\end{example}

\section{Cyclic cohomology of Lie algebras}
In this section we show that for $V$, a SAYD module over a Lie algebra $\Fg$, the (periodic) cyclic cohomology of $\Fg$ with coefficients in $V$ and the (periodic) cyclic cohomology of the enveloping Hopf algebra $U(\Fg)$ with coefficient in the corresponding SAYD over $U(\Fg)$ are isomorphic.

As a result of Proposition \ref{7} and Proposition \ref{23}, we have the following definition.

\begin{definition}
Let $\Fg$ be a Lie algebra and $V$ be a right-left SAYD module over $\Fg$. We call the cohomology of the  total  complex of  $(C^{\bullet}(\Fg,V), \p_{\rm CE} + b_{\rm K})$ the  cyclic   cohomology of the Lie algebra $\Fg$ with coefficients in the SAYD module $V$, and denote it by  $HC^{\bullet}(\Fg,V)$. Similarly we denote its periodic cyclic cohomology by $HP^{\bullet}(\Fg,V)$.
\end{definition}

Our main result in this section is an analogous of Proposition 7 of \cite{ConnMosc98}.

\begin{theorem}\label{g-U(g) spectral sequence}
Let $\Fg$ be a Lie algebra and $V$ be a SAYD module over the Lie algebra $\Fg$. Then the periodic cyclic homology of $\Fg$ with coefficients in  $V$ is the same as the periodic cyclic cohomology of $U(\Fg)$ with coefficients in the corresponding SAYD module $V$ over $U(\Fg)$. In short,
\begin{align}
HC^\bullet(\Fg, V) \cong HC^\bullet(U(\Fg), V)
\end{align}
\end{theorem}

\begin{proof}
The total coboundary of $C(\Fg, V)$ is $\p_{\rm CE} + \p_{\rm K}$ while the total coboundary of the complex $C(U(\Fg), V)$ computing the  cyclic cohomology of $U(\Fg)$ is $B + b$.

Next, we compare the $E_1$ terms of the spectral sequences of the total complexes corresponding to the filtration on the complexes which is induced by the filtration on $V$ via \cite[Lemma 6.2]{JaraStef}. To this end, we first show that the coboundaries respect this filtration.

As it is indicated in the proof of \cite[Lemma 6.2]{JaraStef}, each $F_pV$ is a submodule of $V$. Thus, the Lie algebra homology boundary $\p_{\rm CE}$ respects the filtration. As for $\p_{\rm K}$, we notice for $v \in F_pV$
\begin{align}
\p_{\rm K}(X_1 \wdots X_n \ot v) = v\nsb{-1} \wg X_1 \wdots X_n \ot v\nsb{0}
\end{align}
Since
\begin{align}
\Db(v) = v\snsb{-1} \ot v\snsb{0} = 1 \ot v + v\nsb{-1} \ot v\nsb{0} + \sum_{k \geq 2}\theta_k^{-1}(v\nsb{-k} \cdots v\nsb{-1}) \ot v\nsb{0}
\end{align}
we observe  that $v\nsb{-1} \wg X_1 \wdots X_n \ot v\nsb{0} \in \wedge^{n+1}\Fg \ot F_{p-1}V$.  Since  $F_{p-1}V \subseteq F_pV$, we conclude that  $\p_{\rm K}$ respects the filtration.

Since the Hochschild coboundary $b: V\ot U(\Fg)^{\ot n}\ra  V\ot U(\Fg)^{\ot n+1}$ is the alternating sum of cofaces $\d_i$, it suffices to check each $\d_i$ preserve the filtration, which  is obvious for all cofaces  except possibly   the last one. However,  for the last coface, we take  $v \in F_pV$ and  write
\begin{align}
v\snsb{-1} \ot v\snsb{0} = 1 \ot v + v\ns{-1} \ot v\ns{0}, \qquad v\ns{-1} \ot v\ns{0} \in \Fg \ot F_{p-1}V.
\end{align}
We have
\begin{align}
\delta_n(v \ot u^1 \ot \cdots \ot u^n) = v\snsb{0} \ot u^1 \ot \cdots \ot u^n \ot v\snsb{-1} \in F_pV \ot U(\Fg)^{\ot\, n+1}.
\end{align}
Hence, we can say that $b$ respects the filtration.

For the cyclic operator, the result again follows from the fact that $F_p$ is a $\Fg$-module. Indeed, for $v \in F_pV$
\begin{align}
\tau_n(v \ot u^1 \ot \cdots \ot u^n) = v\snsb{0} \cdot u^1\ps{1} \ot S(u^1\ps{2}) \cdot (u^2 \ot \cdots \ot u^n \ot v\snsb{-1}) \in F_pV \ot U(\Fg)^{\ot\, n}
\end{align}

Finally we consider the extra degeneracy operator
\begin{align}
\sigma_{-1}(v \ot u^1 \ot \cdots \ot u^n) = v \cdot u^1\ps{1} \ot S(u^1\ps{2}) \cdot (u^2 \ot \cdots \ot u^n) \in F_pV \ot U(\Fg)^{\ot\, n}
\end{align}
which  preserves the filtration again by using the fact that $F_p$ is $\Fg$-module and the coaction preserve the filtration. As a result now, we can say that the Connes' boundary $B$ respects the filtration.

Now, the $E_1$-term of the spectral sequence associated to the filtration $(F_pV)_{p \geq 0}$ computing the periodic cyclic cohomology of the Lie algebra $\Fg$ is known to be of the form
\begin{align}
E_1^{j,\,i}(\Fg) = H^{i+j}(F_{j+1}C(\Fg, V)/F_jC(\Fg, V), [\p_{\rm CE} + \p_{\rm K}])
\end{align}
where, $[\p_{\rm CE} + \p_{\rm K}]$ is the induced coboundary operator on the quotient complex. By the obvious identification
\begin{align}
F_{j+1}C(\Fg,V)/F_jC(\Fg,V) \cong C(\Fg,F_{j+1}V / F_jV) = C(\Fg,(V / F_jV)^{co\Fg}),
\end{align}
we observe  that
\begin{align}
E_1^{j,\,i}(\Fg) = H^{i+j}(C(\Fg,(V / F_jV)^{coU(\Fg)}), [\p_{\rm CE}]),
\end{align}
for  $\p_{\rm K}(F_{j+1}C(\Fg,V)) \subseteq F_jC(\Fg,V)$.

 Similarly,
\begin{align}
E_1^{j,\,i}(U(\Fg)) = H^{i+j}(C(U(\Fg),(V / F_jV)^{coU(\Fg)}), [b + B]).
\end{align}
Finally,  considering
\begin{align}
E_1^{j,\,i}(\Fg) = H^{i+j}(C(\Fg,(V / F_jV)^{co\Fg}), [0] + [\p_{\rm CE}])
\end{align}
\ie as a bicomplex with degree +1 differential is zero, the anti-symmetrization map $\alpha:C(\Fg,(V / F_jV)^{co\Fg}) \to C(U(\Fg),(V / F_jV)^{coU(\Fg)})$ induces a quasi-isomorphism $[\alpha]:E_1^{j,\,i}(\Fg) \to E_1^{j,\,i}(U(\Fg))$, $\forall i,j$ by Proposition 7 in \cite{ConnMosc98}.
\end{proof}

\begin{remark}
In case the $\Fg$-module $V$ has a trivial $\Fg$-comodule structure, the coboundary $\p_{\rm K} = 0$ and
\begin{align}
HP^{\bullet}(\Fg, V) = \bigoplus_{n = \bullet \, mod \, 2} H_n(\Fg, V).
\end{align}
In this case, the above theorem becomes \cite[Proposition 7]{ConnMosc98}.
\end{remark}

\bibliographystyle{amsplain}
\bibliography{Rangipour-Sutlu-References}{}

\providecommand{\bysame}{\leavevmode\hbox to3em{\hrulefill}\thinspace}
\providecommand{\MR}{\relax\ifhmode\unskip\space\fi MR }
% \MRhref is called by the amsart/book/proc definition of \MR.
\providecommand{\MRhref}[2]{%
  \href{http://www.ams.org/mathscinet-getitem?mr=#1}{#2}
}
\providecommand{\href}[2]{#2}
\begin{thebibliography}{10}

\bibitem{AlekMein}
A.~Alekseev and E.~Meinrenken, \emph{The non-commutative {W}eil algebra},
  Invent. Math. \textbf{139} (2000), no.~1, 135--172. \MR{1728878
  (2001j:17022)}

\bibitem{ChevEile}
Claude Chevalley and Samuel Eilenberg, \emph{Cohomology theory of {L}ie groups
  and {L}ie algebras}, Trans. Amer. Math. Soc. \textbf{63} (1948), 85--124.
  \MR{0024908 (9,567a)}

\bibitem{ConnMosc98}
A.~Connes and H.~Moscovici, \emph{Hopf algebras, cyclic cohomology and the
  transverse index theorem}, Comm. Math. Phys. \textbf{198} (1998), no.~1,
  199--246. \MR{1657389 (99m:58186)}

\bibitem{Conn83}
Alain Connes, \emph{Cohomologie cyclique et foncteurs {${\rm Ext}^n$}}, C. R.
  Acad. Sci. Paris S\'er. I Math. \textbf{296} (1983), no.~23, 953--958.
  \MR{777584 (86d:18007)}

\bibitem{ConnMosc00}
Alain Connes and Henri Moscovici, \emph{Cyclic cohomology and {H}opf algebra
  symmetry [1800488]}, Conf\'erence {M}osh\'e {F}lato 1999, {V}ol. {I}
  ({D}ijon), Math. Phys. Stud., vol.~21, Kluwer Acad. Publ., Dordrecht, 2000,
  pp.~121--147. \MR{1805887}

\bibitem{Dixm}
Jacques Dixmier, \emph{Enveloping algebras}, Graduate Studies in Mathematics,
  vol.~11, American Mathematical Society, Providence, RI, 1996, Revised reprint
  of the 1977 translation. \MR{1393197 (97c:17010)}

\bibitem{GuilSter}
Victor~W. Guillemin and Shlomo Sternberg, \emph{Supersymmetry and equivariant
  de {R}ham theory}, Mathematics Past and Present, Springer-Verlag, Berlin,
  1999, With an appendix containing two reprints by Henri Cartan [ MR0042426
  (13,107e); MR0042427 (13,107f)]. \MR{1689252 (2001i:53140)}

\bibitem{HajaKhalRangSomm04-II}
Piotr~M. Hajac, Masoud Khalkhali, Bahram Rangipour, and Yorck Sommerh{\"a}user,
  \emph{Hopf-cyclic homology and cohomology with coefficients}, C. R. Math.
  Acad. Sci. Paris \textbf{338} (2004), no.~9, 667--672. \MR{2065371
  (2005b:19002)}

\bibitem{HajaKhalRangSomm04-I}
\bysame, \emph{Stable anti-{Y}etter-{D}rinfeld modules}, C. R. Math. Acad. Sci.
  Paris \textbf{338} (2004), no.~8, 587--590. \MR{2056464 (2005a:16056)}

\bibitem{JaraStef}
P.~Jara and D.~{\c{S}}tefan, \emph{Hopf-cyclic homology and relative cyclic
  homology of {H}opf-{G}alois extensions}, Proc. London Math. Soc. (3)
  \textbf{93} (2006), no.~1, 138--174. \MR{2235945 (2007g:16013)}

\bibitem{Kass}
Christian Kassel, \emph{Quantum groups}, Graduate Texts in Mathematics, vol.
  155, Springer-Verlag, New York, 1995. \MR{1321145 (96e:17041)}

\bibitem{Knap}
Anthony~W. Knapp, \emph{Lie groups, {L}ie algebras, and cohomology},
  Mathematical Notes, vol.~34, Princeton University Press, Princeton, NJ, 1988.
  \MR{938524 (89j:22034)}

\bibitem{KumaVerg}
Shrawan Kumar and Mich{\`e}le Vergne, \emph{Equivariant cohomology with
  generalized coefficients}, Ast\'erisque (1993), no.~215, 109--204, Sur la
  cohomologie {\'e}quivariante des vari{\'e}t{\'e}s diff{\'e}rentiables.
  \MR{1247061 (95f:22019)}

\bibitem{Maji}
Shahn Majid, \emph{Foundations of quantum group theory}, Cambridge University
  Press, Cambridge, 1995. \MR{1381692 (97g:17016)}

\bibitem{RangSutl}
Bahram Rangipour and Serkan S\"utl\"u, \emph{Lie-hopf algebras and their
  hopf-cyclic cohomology}, arXiv.org:math/1012.4827.

\bibitem{Wall}
Nolan~R. Wallach, \emph{Invariant differential operators on a reductive {L}ie
  algebra and {W}eyl group representations}, J. Amer. Math. Soc. \textbf{6}
  (1993), no.~4, 779--816. \MR{1212243 (94a:17014)}

\end{thebibliography}

\end{document}